\documentclass[11pt,a4paper]{article}
\usepackage{amsmath,amsthm,amssymb,amsfonts,mathrsfs,hyperref,graphicx,multirow,ulem,bm} %
\usepackage[left=2.5cm,right=2.5cm,top=2cm,bottom=2cm]{geometry}
\usepackage{indentfirst} 

\newtheorem{theorem}{Theorem}[section]
\newtheorem{corollary}{Corollary}[section]
\newtheorem{definition}{Definition}[section]
\newtheorem{lemma}{Lemma}[section]
\newtheorem{proposition}{Proposition}[section]
\newtheorem{remark}{Remark}[section]

\numberwithin{equation}{section}

\allowdisplaybreaks[4]
\DeclareMathOperator*{\dist}{dist}

\renewcommand{\d}{\, \mathrm{d}}
\newcommand{\wtK}{\widetilde{K}}
\newcommand{\Phie}{\Phi^{\varepsilon}}
\newcommand{\cA}{\mathcal{A}}
\newcommand{\cD}{\mathcal{D}}
\newcommand{\cH}{\mathcal{H}}
\newcommand{\cM}{\mathcal{M}}
\newcommand{\cO}{\mathcal{O}}
\newcommand{\bH}{\bm{H}}
\newcommand{\bV}{\bm{V}}
\newcommand{\bP}{\mathbb{P}}
\newcommand{\bR}{\mathbb{R}}
\newcommand{\pt}{\partial}
\newcommand{\vn}{v_{n}}
\newcommand{\etan}{\eta_{n}^{t}}
\newcommand{\psin}{\psi_{n}}
\newcommand{\hpsi}{\hat{\psi}}
\newcommand{\vN}{v_{N}}
\newcommand{\vL}{v_{L}}
\newcommand{\etaL}{\eta_{L}}
\newcommand{\etaN}{\eta_{N}}
\newcommand{\psiN}{\psi_{N}}
\newcommand{\psiL}{\psi_{L}}
\newcommand{\phiN}{\phi_{N}}
\newcommand{\phiL}{\phi_{L}}
\newcommand{\abs}[1]{\left\vert #1 \right\vert}
\newcommand{\norm}[1]{\left\Vert #1 \right\Vert}
\newcommand{\normM}[1]{\left\Vert #1 \right\Vert_{\cM}}
\newcommand{\inner}[2]{\left\langle #1 , #2 \right\rangle}
\newcommand{\innerM}[2]{\left\langle #1 , #2 \right\rangle _{\mathcal{M}}}
\newcommand{\innerMs}[2]{\left\langle #1 , #2 \right\rangle _{\mathcal{M}^{1}}}
\newcommand{\tran}[2]{\left( #1 , #2 \right)^{\tau}}
\newcommand{\onehalf}{\frac{1}{2}}

\begin{document}
\setlength{\baselineskip}{16pt}

\title{
	\textbf{Asymptotic behavior for
	2D stochastic Navier-Stokes equations with memory in unbounded domains}
}

\author{Yadong Liu$^1$, Wenjun Liu$^1$\footnote{Corresponding author. \ \   Email address: wjliu@nuist.edu.cn (W. J. Liu).}, Xin-Guang Yang $^2$\ and\ Yasi Zheng$^1$ \medskip\\
	1. School of Mathematics and Statistics, Nanjing University \\ of Information Science and Technology, Nanjing 210044, China
	\medskip\\
	2. College of Mathematics and Information Science, Henan Normal University, \\
	Xinxiang 453007, China
}    

\date{\today}
\maketitle

\begin{abstract}
We consider a stochastic model which describes the motion of a 2D incompressible fluid in a unbounded domain with viscosity and memory effects. This model is different from the classical stochastic Navier-Stokes-Voigt equations due to  the absence of the Voigt term $ -\alpha \Delta u_{t}$, and has a much weaker dissipation than the usual Navier-Stokes-Voigt model since only the memory viscoelasticity is present. We are interested in the global well-posedness and long-time behaviors of this model. We first investigate the well-posedness by using the classical Faedo-Galerkin method. Unlike the general method of energy estimate, we then split the solution into two parts and get the low-order and high-order uniform estimates, respectively. Based on the uniform estimates of far-field values of solutions, we further prove the existence and uniqueness of random attractors in unbounded domains with a constructed compact subspace corresponding to memory. Finally, we give the upper semicontinuity of the attractors when stochastic perturbation approaches to zero.
\end{abstract}

\noindent {\it 2010 Mathematics Subject Classification:} 35B40, 35B41, 35R60, 37L55 . \\
\noindent {\bf Keywords:} Navier-Stokes equations, well-posedness, random attractor, memory effects, semicontinuity.

\maketitle

\section{Introduction}
This paper is concerned with the asymptotic behavior of stochastic Navier-Stokes equations with memory in unbounded domains in $ \bR^{2} $. Let $ \cO $ be an arbitrary domain (bounded or unbounded) in $ \bR^2 $, in which the Poincar\'{e} inequality holds
\begin{equation*}
	\int_{ \cO } \left \vert \nabla u \right \vert ^2 \d x \geq \lambda_1 \int_{ \cO } \left \vert u \right \vert ^2 \d x, \quad \forall \ u \in [ H_{0}^{1}(\cO) ]^2.
\end{equation*}
For $ t > 0 $, we consider the stochastic Navier-Stokes equations with memory effect
\begin{equation}\label{SNSM}
\left\{
	\begin{aligned}
		&
		\begin{aligned}
			& \pt_t u - \nu \Delta u - \int_{0}^{\infty} g(s) \Delta u ( t - s ) \d s  \\
			& \qquad \qquad \qquad \qquad + ( u \cdot \nabla ) u + \nabla p = f + \varepsilon h \frac{\d W}{\d t},
		\end{aligned}
		&& x \in \cO, \ t > 0, \\
		& \nabla \cdot u = 0, && x \in \cO, \ t > 0, \\
		& u( x , t ) = 0, && x \in \pt \cO, \ t > 0, \\
		& u( x , 0 ) = u_0(x), u( x , - s ) = \rho ( x , s ), && x \in \cO, \ s > 0,
	\end{aligned}
\right.
\end{equation}
where $ u = u ( x , t ) $ and $ p = p ( x , t ) $ are the unknown velocity and pressure respectively, while $ \nu $ denotes the positive viscosity coefficient, $ f $ and $ h $ are two given functions which will be declared later, $ \varepsilon \in (0,1] $ represents a small parameter. $ W $ stands for a two-side real-value Wiener process on a complete probability space which will be specified later. Here $ u_0(x) $ and $ \rho(x,s) $ denotes the prescribed data of initial velocity and past history for the purpose of completing system \eqref{SNSM}. Concerning the kernel $ g : [ 0 , \infty ) \rightarrow \bR $, we assume that it is convex, nonnegative, and smooth on $ \bR^{+} = ( 0 , \infty ) $. Also it is supposed to satisfy
$$
	\lim\limits_{s \rightarrow \infty} g(s) = 0 \quad \mbox{ and } \quad \int_{0}^{\infty} g( s ) \d s = 1.
$$

For deterministic case, Oskolkov \cite{Oskolkov1973} first studied the incompressible fluid with Kelvin-Voigt viscoelasticity which was illustrated by Navier-Stokes-Voigt system
\begin{equation}\label{NSV}
	\left\{
	\begin{aligned}
		& \pt_t u - \Delta u - \alpha \Delta (\pt_t u) + ( u \cdot \nabla ) u + \nabla p = f, \\
		& \nabla \cdot u =0.
	\end{aligned}
	\right.
\end{equation}
The existence of finite dimensional global attractors was investigated by Kalantarov and Titi in \cite{KT2009} and Anh and Trang \cite{AT2013} showed the existence of a weak solution to the problem by using the Faedo-Galerkin method in unbounded domains. After that, many authors considered system \eqref{NSV} in different aspects. Readers are referred to \cite{SQ2018,Yang2019,You2017,Zhou2018} and references therein.

Memory term arose in the description of several phenomena like, e.g., heat conduction in special materials (see e.g., \cite{Caraballo2007,GP2000,Liu2017a,Liu2017b}), viscoelasticity of vibration in several materials (see e.g., \cite{F2018,LZ2017,Messaoudi2008}). Actually, the presence of the memory destroys the parabolic character of the system and provides a more realistic description of the viscosity while $ - \nu \Delta u $ indicates the instantaneous viscous effect. Astarita and Marucci \cite[pp. 132]{AM1974} induced a first-order approximation to the constitutive equation of a simple fluid with fading memory
\begin{equation}\label{constitutive}
	\mathbb{T} = -p \mathbb{I} + \int_{0}^{\infty} f(s) G^{t} \d s,
\end{equation}
where $ \mathbb{T} $ is the stress and $ p $ is the pressure; $ G^{t} $ denotes the deformation history at some instant of observation $ t $; $ f(\cdot) $ is characteristic of the particular material. The constitutive equation \eqref{constitutive} was called ``linear viscoelasticity'' by the authors.
In 2005, Gatti, Giorgi and Pata \cite{GGP2005} proposed a Jeffreys type model depicting the motion of a   two dimensional viscoelastic polymeric fluid with memory effect of the form
\begin{equation}\label{NSMJ}
	\left\{
	\begin{aligned}
		& \pt_t u - \nu \Delta u - ( 1 - \nu ) \int_{0}^{\infty} k_{\varepsilon} (s) \Delta \eta(s) \d s + (u \cdot \nabla) u + \nabla p = f, \\
		& \pt_t \eta = - \pt_{s} \eta + u, \\
		& \nabla \cdot u = 0,\ \nabla \cdot \eta = 0,
	\end{aligned}
	\right.
\end{equation}
where $ \nu \in ( 0 , 1 ) $ is a fixed parameter. In the system, the so-called memory kernel is defined as
$$
	k_{\varepsilon}(s) = \frac{1}{\varepsilon ^2} k \left( \frac{s}{\varepsilon} \right), \quad \varepsilon \in ( 0 , 1 ].
$$
They described the asymptotic dynamics and proved that when the scaling parameter $ \varepsilon $ in the memory kernel (physically, the Weissenberg number of the flow) tends to zero, the model converges to the Navier-Stokes equations in an appropriate sense.
More recently, Gal and Medjo \cite{GM2013} put forward the following NSV system incorporating hereditary effects by adding the Voigt term $ - \alpha \Delta u _{t} $ in \eqref{NSMJ}:
\begin{equation}\label{NSVM}
	\left\{
	\begin{aligned}
		& \pt_t u - \nu \Delta u - ( 1 - \nu ) \int_{0}^{\infty} k_{\varepsilon} (s) \Delta \eta(s) \d s - \alpha \Delta (\pt_t u) + ( u \cdot \nabla ) u + \nabla p = f, \\
		& \pt_t \eta = - \pt_{s} \eta + u, \\
		& \nabla \cdot u = 0,\ \nabla \cdot \eta = 0,
	\end{aligned}
	\right.
\end{equation}
for some $ \nu \in ( 0 , 1 ) $. They took both Newtonian contributions and viscoelastic effects into account and considered a Cauchy stress tensor to derive the model. By the fact that the coupled effects of instantaneous viscous term $ \Delta u $, Voigt term $ \Delta (\pt_t u) $ and hereditary kinematic viscous term $ \int_{0}^{\infty} g(s) \Delta u( t - s ) \d s $ are strong enough to stabilize the system, Di Plinio, Giorgini, Pata and Temam \cite{PGPT2018} investigated the long-time behavior of the following system without $ \Delta u$:
\begin{equation}\label{NSM}
	\left\{
	\begin{aligned}
		& \pt_t u -  \int_{0}^{\infty} g(s) \Delta u( t - s ) \d s - \alpha \Delta (\pt_t u) + ( u \cdot \nabla ) u + \beta u + \nabla p = f, \\
		& \nabla \cdot u =0,
	\end{aligned}
	\right.
\end{equation}
where $ \beta u $ is the Ekman term. They showed that the solution of \eqref{NSM} decays exponentially if $ f \equiv 0 $ and the system is dissipative from the viewpoint of dynamical systems. What's more, they concluded that the system also possesses regular global and exponential finite fractal dimensional attractors.

All the external forcing terms $ f $ above are deterministic, but actually, it is more meaningful that the system meets different random perturbations. So people considered different stochastic effects and combined the theory of random dynamics. Since 90s last century, there were many researches on stochastic Navier-Stokes equations.
Flandoli and Schmalfuss \cite{FS1996} first combined random dynamical theory and Navier-Stokes equations and showed that there exist random attractors for 3D stochastic Navier-Stokes equations on bounded domains. However, they did not give the existence and uniqueness of solutions. Mar\'{i}n-Rubio and Robinson \cite{MR2003} investigated the attractors for a 3D stochastic Navier-Stokes equations with additive white noise by a generalized semiflow.
Bre\'{z}niak and Li \cite{BL2006} proved the existence of the stochastic flow associated with 2D stochastic Navier-Stokes equations in possibly unbounded Poincar\'{e} domains by the classical Galerkin approximation. They then deduced the existence of an invariant measure for such system in two dimensional case.
In \cite{BCLL2013}, Bre\'{z}niak, Carabollo, Lange and Li completed the result in \cite{BL2006}. They considered the random attractors of 2D Navier-Stokes equations in some unbounded domains and showed that the stochastic flow generated by the 2-dimensional Stochastic Navier–Stokes equations with rough noise on a Poincar\'{e}-like domain has a unique random attractor.
It is known that the uniqueness of Navier-Stokes equations in three dimensional case is still open. So people turned to 3D Navier-Stokes-Voigt equations, in which it has the term $ -\alpha \Delta (\pt_t u) $, resulting in the uniqueness of the system.
Gao and Sun \cite{Gao2012} examined the well-posedness of the 3D Navier-Stokes-Voigt system by classical Faedo-Galerkin method. They also investigated the random random attractors of three dimensional stochastic Navier-Stokes-Voigt equations. Further, Bao \cite{Tang2014} continued corresponding works in unbounded domains. They use the so-called energy equation method, which was introduced by Ball \cite{Ball2004} to establish the existence of random attractors. By means of the method in \cite{Wang2009}, they proved the upper semicontinuity of random attractors. For other works on stochastic Navier-Stokes-Voigt system please see \cite{AT2019,LS2018,Yang2018} and the references therein.

However, studies on stochastic Navier-Stokes equations with memory is still lack. Motivated by the literature above, we investigate the well-posedness and asymptotic behaviors of two dimensional stochastic Navier-Stokes equations \eqref{SNSM} in unbounded domains in this paper. More precisely, compared to \eqref{NSM}, we have the viscosity dissipation, memory effects and random perturbation, while we do not have $ - \Delta (\pt_t u) $ (Voigt term) and $ \beta u $ (Ekman term). The main features of our work are summarized as follows.
\begin{enumerate}\setlength{\itemsep}{-0.2cm}
	\item[(i)] Notice that Sobolev embeddings are no longer compact in unbounded domains. It leads to a major difficulty for us to prove the asymptotic compactness of solutions by standard method. To overcome this difficulty, we refer to \cite{BLW2009,Wang1999} which provide uniform estimates on the far-field values of solutions. Moreover, we establish a generalized Poincar\'{e} inequality to construct the weighted energy since we do not have $ \beta u $ in the system.
	\item[(ii)] The procedure in \cite{BLW2009} indicates that we still need the compact embedding from higher regular space to common space in a bounded ball. Due to the memory term, the common regular space is $ \cH = \bm{H} \times \mathcal{M} $ and higher regular space is $ \cH^{1} = \bm{V} \times \mathcal{M}^{1} $. Though the embedding $ \bm{V} \hookrightarrow \bm{H} $ is compact, we can't say that the embedding $ \mathcal{M}^{1} \hookrightarrow \mathcal{M} $ is also compact. Nevertheless, we can recover the compactness with methods in \cite{Liu2017a,PataZ2001}, for which we introduce a compact subspace $ \overline{\mathcal{N}} \subset \mathcal{M} $ and obtain a compact embedding $ \widetilde{\cH} \hookrightarrow \cH $ in a bounded ball.
	\item[(iii)] In this manuscript, $ \psi_{0} \in \cH $, so we can not obtain the higher order estimate by using classical energy method. To this end, we split the system into a ``linear'' system and a zero initial data nonlinear system \cite{GP2000,KT2009,Liu2017a,Temam1988}. The energy of the ``linear'' system decays exponentially to 0 in $ \cH $ while the energy of nonlinear system is bounded in $ \cH^{1} $. Then we can using this property to deduce the compactness.
\end{enumerate}
Symbols above are all assigned in Section \ref{notation}.

The paper is arranged as follows. In Section \ref{notation}, we recall the relevant mathematical framework for Navier-Stokes equations and memory kernel. In Section \ref{random}, we take some fundamental results on the existence and semicontinuity of pullback random attractors for random dynamical systems, also we show that \eqref{SNSM} generates a random dynamical system by several transformations. To derive the global well-posedness, we use classical Faedo-Galerkin method in Section \ref{wellposd}. Some necessary uniform \textit{a prior} and far-field estimates are proposed in Section \ref{attractor}. By means of solution splitting method, far-field estimates and a compact embedding we construct, we then prove the existence and uniqueness of random attractor for \eqref{SNSM} in Section \ref{attractor}. In Section \ref{upper}, we further show the upper semicontinuity of the attractors when the stochastic perturbation parameters $ \varepsilon $ tends to zero. As usual, letter $ c $ in the paper represents generic positive constant which may change its value from line to line or even in the same line, unless we give a special declaration.

\section{Mathematical Settings and Notations}\label{notation}
In this section, we present some mathematical settings and notations as what in \cite{PGPT2018}. $ L^2(\cO) $, $H_0^1(\cO)$, $ H ^r(\cO) $ are standard Lebesgue-Sobolev spaces in $ \cO $.
Let
$$
	\mathcal{V} = \left\{ u \in \left[ C_0^{\infty}(\cO)\right]^2:\ \nabla \cdot u = 0 \right\}.
$$
We denote by $ \bm{H} $ the completion of $ \mathcal{V} $ in the norm of $ \left[ L^2(\cO) \right]^2 $ and by $ \bm{V} $ the completion of $ \mathcal{V} $ in the norm of $ \left[ H_0^1(\cO) \right]^2 $. Inner product and norm of $ \bm{H} $ and $ \bm{V} $ are
$$
	\inner{u}{v}= ( u , v ) \quad \mbox{and} \quad \Vert u \Vert = ( u , u ),
$$
and
$$
	\inner{u}{v} _{1} = \inner{\nabla u}{\nabla v} \quad \mbox{and} \quad \Vert u \Vert _{1} = \Vert \nabla u \Vert,
$$ respectively.
We also denote by $ \bm{H}' $ the dual space of $ \bm{H} $ and by $ \bm{V}' $ the dual space of $ \bm{V} $. It follows that
$ \bm{V} \subset \bm{H} \equiv \bm{H}' \subset \bm{V}'$,
where the injections are dense and continuous. We define the more regular space by
$$
	\bm{W} = \bm{V} \cap \left[ H^2 (\cO) \right]^2.
$$

Recalling the Leray orthogonal projection $ \mathcal{P} : \left[ L^2(\cO) \right]^2 \rightarrow \bm{H} $ from \cite{lions1969,Temam1988,Temam1995}, we take the Stokes operator $ A $ on $ \bm{H} $ by
$
	A = - \mathcal{P} \Delta \mbox{ with domain } \mathfrak{D} (A) = \bm{W}.
$
Then $ A $ is a positive self-adjoint operator with compact inverse and $ ( A u , u ) = ( A^{\frac{1}{2}} u , A^{\frac{1}{2}} u ) $ for all $ u \in \bm{V} $ (see e.g., \cite{Temam2001}). Hence, we define the compactly nested Hilbert spaces
$$
	V^{r} = \mathfrak{D} ( A^{\frac{r}{2}} ), \quad r \in \bR,
$$
endowed with inner product and norm
$$
	\inner{u}{v} _{r} = \inner{A^{\frac{r}{2}} u}{A^{\frac{r}{2}} v} \quad \mbox{and} \quad \Vert u \Vert _{r} = \Vert A^{\frac{r}{2}} u \Vert.
$$

As usual, we define the continuous trilinear form $ b $ on $ \bm{V} \times \bm{V} \times \bm{V} $ by
$$
	b( u , v , w) = \int_{ \cO } ( u \cdot \nabla ) v \cdot w \d x = \sum_{i,j = 1}^{2} \int_{ \cO } u_{i} \frac{\pt v_{j}}{\pt x_{i}} w_{j} \d x.
$$
By integration by parts, we easily prove that
$$
	b( u , v , w) = - b( u , w , v),
$$
and hence
$$
	b( u , v , v) = 0, \quad \forall \ u \in \bm{V}, v \in \left[ H_0^1(\cO) \right]^2.
$$
The bilinear form $ B : \bm{V} \times \bm{V} \rightarrow \bm{V}' $ is defined as
$$
	\inner{B( u , v )}{w} = b( u , v , w).
$$

Then we introduce the common estimates for trilinear form $ b( u , v , w) $.
\begin{lemma}[see e.g., \cite{FMRT2001,RRS2016,Temam1988,Temam1995,Temam2001}]
	\label{b}
	For all $ u, v, w \in \bm{V} $, we have
	\begin{equation*}
		\begin{aligned}
			& \abs{b(u,v,w)} \leq \hat{c} \norm{u}^\onehalf \norm{A^\onehalf u}^\onehalf
			\norm{v}^\onehalf \norm{A^\onehalf v}^\onehalf \norm{A^\onehalf w}; \\
			& \abs{b(u,v,w)} \leq \hat{c} \norm{u}^\onehalf \norm{A^\onehalf u}^\onehalf
			\norm{A^\onehalf v} \norm{w}^\onehalf \norm{A^\onehalf w}^\onehalf.
		\end{aligned}
	\end{equation*}
\end{lemma}

In what follows, we describe the mathematical framework with respect to memory term. The function $ g $ is supposed to have the explicit form
$$
\mu (s) = - g ' (s).
$$
We assume that $ \mu $ here is nonnegative, absolutely continuous and decreasing. Hence $ \mu ' \leq 0 $ for almost every $ s \in \bR^{+} $. Moreover, $ \mu $ is summable on $ \bR^{+} $ with
$$
\kappa = \int_{0}^{\infty} \mu (s) \d s > 0.
$$
In our work, we consider the classical Dafermos condition (see e.g., \cite{Dafermos1970})
\begin{equation}\label{mu}
\mu ' (s) + \delta \mu (s) \leq 0,
\end{equation}
for some $ \delta > 0 $ and almost every $ s > 0 $. Then we define the weighted Hilbert space for memory on $ \bR^{+} $
$$
	\mathcal{M} = L_{\mu}^{2} ( \bR^{+} , \bm{V} ),
$$
endowed with inner product and norm
$$
	\innerM{\eta}{\xi} = \int_{0}^{\infty} \mu (s) \inner{\eta (s)}{\xi (s)} _{1} \d s
	\ \mbox{  and  } \
	\Vert \eta \Vert _{\mathcal{M}} = \left( \int_{0}^{\infty} \mu (s) \Vert \eta (s) \Vert _{1}^{2} \d s \right)^{\frac{1}{2}}.
$$
The infinitesimal generator of the right-translation semigroup on $ \mathcal{M} $ is the linear operator
$$
	T \eta = - \pt _{s} \eta
	\mbox{ with domain }
	\mathfrak{D} (T) = \left\{ \eta \in \mathcal{M} : \pt _{s} \eta \in \mathcal{M}, \ \eta (0) = 0 \right\},
$$
where $ \pt _{s} \eta $ stands for the derivative of $ \eta (s) $ in regard to $ s $.

In the end, we introduce the phase space
$$
	\cH = \bm{H} \times \mathcal{M},
$$
endowed with norm
$$
	\Vert( u , \eta ) \Vert _{\cH}^2 = \Vert u \Vert ^2 + \Vert \eta \Vert _{\mathcal{M}}^2.
$$
In this paper, we also utilize a more regular memory space denoted by
$$
	\mathcal{M}^{1} = L_{\mu}^{2} ( \bR^{+} , \bm{W} ),
$$
with norm analogous to that of $ \mathcal{M} $. What's more, the related higher order phase space is denoted by
$$
	\cH^{1} = \bm{V} \times \mathcal{M}^{1}
$$
with norm
$$
	\Vert( u , \eta ) \Vert _{\cH^{1}}^2 = \Vert u \Vert _{1}^2 + \Vert \eta \Vert _{\mathcal{M}^{1}}^2.
$$

\section{Random Dynamical System}\label{random}
\subsection{Random attractors} \label{randomdy}
In this subsection, we recall some basic concepts on the theory of random attractors for random dynamical systems. For a piece of detailed information and related applications, readers are referred to \cite{Arnold1998,BLW2009,Wang2009,Wang2014}.

Let $ ( X , \Vert \cdot \Vert _{X} ) $ be a separable Banach space with the Borel $ \sigma $-algebra $ \mathcal{B} (X) $ and $ ( \Omega , \mathcal{F} , \bP ) $ be a probability space.
\begin{definition}
	$ ( \Omega , \mathcal{F} , \bP , ( \theta_{ t } )_{t \in \bR} ) $ is said to be a \textbf{metric dynamical system} if $ \theta : \bR \times \Omega \rightarrow \Omega $ is $ ( \mathcal{B} ( \bR ) \times \mathcal{F} , \mathcal{F} ) $-measurable and satisfies that $ \theta_{0} $ is the identity on $ \Omega $, $ \theta_{ t + s } = \theta_{ t } \circ \theta_{ s } $ for all $ t , s \in \bR $ and $ \theta_{ t } (\bP) = \bP $ (measure preserved) for all $ t \in \bR $. Here $ \circ $ means composition.
\end{definition}
\begin{definition}
	\label{RDS}
	A mapping
	$$
	\Phi : \bR^{+} \times \Omega \times X \rightarrow X,
	\quad
	( t , \omega , x ) \mapsto \Phi ( t , \omega ) x,
	$$ is known as a \textbf{random dynamical system} over a metric dynamical system $ ( \Omega , \mathcal{F} , \bP , ( \theta_{ t } )_{t \in \bR} ) $ if for $ \bP $-a.e. $ \omega \in \Omega $,
	\begin{enumerate}
		\item[(i)] $ \Phi ( 0 , \omega )  = \mathrm{Id}_{ X } $ on $ X $;
		\item[(ii)] $ \Phi ( t + s , \omega ) = \Phi ( t , \theta_{ s } \omega ) \circ \Phi ( s , \omega ) $, for all $ t , s \in \bR^{+} $ (cocycle property).
	\end{enumerate}
	A random dynamical system $ \Phi $ is continuous if $ \Phi ( t , \omega ) : X \rightarrow X $ is continuous for all $ t \in \bR^{+}$, $ \omega \in \Omega $.
\end{definition}

\begin{definition}
	A bounded random set is a \textbf{random set} $ B : \Omega \rightarrow 2 ^{X} $ which satisfies that there is a random variable $ r( \omega ) \in [ 0 , \infty ) $, $ \omega \in \Omega $, such that
	$$
		d( B( \omega ) ) := \sup \left\{ \Vert x \Vert _{X} : x \in B( \omega ) \right\} \leq r( \omega ) \mbox{ for all } \omega \in \Omega.
	$$
	A bounded random set $ \left\{ B( \omega ) \right\} _{ \omega \in \Omega } $ is said to be tempered in regard to the metric dynamical system $ ( \Omega , \mathcal{F} , \bP , ( \theta_{ t } )_{t \in \bR} ) $ if for $ \bP $-a.e. $ \omega \in \Omega $,
	$$
		\lim\limits_{ t \rightarrow \infty } e ^{ - \mu t } d(B( \theta_{ - t } \omega )) = 0 \mbox{ for all } \mu > 0.
	$$
\end{definition}

In this manuscript, $ \cD $ always denotes the collection of random sets of $ \cH $, i.e.,
$$
	\cD = \biggl\{ B = \{B(\omega)\}_{\omega \in \Omega}: B(\omega) \subseteq \cH (\cO) \mbox{ and } B \mbox{ is tempered } \biggr\}.
$$
\begin{definition}
	A random set $ \left\{ K( \omega ) \right\} _{ \omega \in \Omega } \in \cD $ is defined as a \textbf{random absorbing set} for $ \Phi $ in $ \cD $ if for every $ B \in \cD $ and $ \bP $-a.e. $ \omega \in \Omega $, there is a $ T ( B , \omega ) > 0 $ such that
	$$
		\Phi ( t , \theta_{ - t } \omega ) B( \theta_{ - t } \omega ) \subseteq K( \omega ), \mbox{ for all } t \geq T ( B , \omega ).
	$$
\end{definition}
\begin{definition}
	A random dynamical system $ \Phi $ is ($ \cD $-pullback) \textbf{asymptotically compact} in $ X $ if for $ \bP $-a.e. $ \omega \in \Omega $, $ \left\{ \Phi ( t_{n} , \theta_{ - t_{n} } \omega ) x_{n} \right\}_{ n = 1 }^{ \infty } $ has a convergent subsequence in $ X $ whenever $ t_{n} \rightarrow \infty $, and $ x_{n} \in B( \theta_{ - t_{n} } \omega ) $ with $ \left\{ B( \omega ) \right\} _{ \omega \in \Omega } \in \cD $.
\end{definition}
\begin{definition}
	A $ \cD $-\textbf{pullback attractor} for $ \Phi $ is a random set $ \left\{ \cA( \omega ) \right\} _{ \omega \in \Omega } $ of $ X $ which satisfied that for $ \bP $-a.e. $ \omega \in \Omega $,
	\begin{enumerate}
		\item[(i)] $ \cA( \omega ) $ is compact, and $ \omega \mapsto d( x, \cA(\omega) ) $ is measurable for every $ x \in X $;
		\item[(ii)] $ \left\{ \cA( \omega ) \right\} _{ \omega \in \Omega } $ is invariant, i.e.,
		$$
			\Phi ( t , \omega ) \cA( \omega ) = \cA( \theta_{ t } \omega )
			\mbox{ for all }
			t \geq 0;
		$$
		\item[(iii)] $ \left\{ \cA( \omega ) \right\} _{ \omega \in \Omega } $ attracts every set in $ \cD $, i.e., for every $ B = \left\{ B( \omega ) \right\} _{ \omega \in \Omega } \in \cD $,
		$$
			\lim\limits_{t \rightarrow \infty} \dist ( \Phi ( t , \theta_{ - t } \omega ) B( \theta_{ - t } \omega ), \cA ( \omega ) ) = 0,
		$$
	\end{enumerate}
	where $ \dist ( \cdot , \cdot ) $ is the Hausdorff semi-distance denfined on $ X $, i.e., for two nonempty sets $ Y , Z \subseteq X $,
	$$
		\dist ( Y , Z ) = \sup_{ y \in Y } \inf_{ z \in Z } \Vert y - z \Vert _{X}.
	$$
\end{definition}
Recall that a collection $ \cD $ of random sets in $ X $ is said to be inclusion-closed if $ E( \omega ) _{ \omega \in \Omega } $ belong to $ \cD $ when $ E( \omega ) _{ \omega \in \Omega } $ is a random set, and $ F( \omega ) _{ \omega \in \Omega } $ belongs to $ \cD $ with $ E( \omega ) \subset F( \omega ) $ for all $ \omega \in \Omega $. The following proposition with respect to the existence and uniqueness of random attractor can be found in \cite{CLR1998,CF1994,FS1996,Wang2009,Wang2014}.

\begin{proposition}\label{attractors}\cite{Wang2009}
	Let $ \cD $ be an inclusion-closed collection of random subsets of $ X $. Assume that $ \left\{ K( \omega ) \right\} _{ \omega \in \Omega } $ is a closed random absorbing set for $ \Phi $ in $ \cD $ and $ \Phi $ is $ \cD $-pullback asymptotically compact in X. Then $ \Phi $ has a unique $ \cD $-\textbf{random attractor} $ \left\{ \cA( \omega ) \right\} _{ \omega \in \Omega } $ given by
	$$
		\cA( \omega ) = \bigcap_{ \tau \geq 0 } \overline{ \bigcup_{ t \geq \tau } \Phi( t , \theta_{ - t } \omega ) K( \theta_{ - t } \omega ) }, \mbox{ for every } \omega \in \Omega.
	$$
\end{proposition}
\subsection{Upper semicontinuity of random attractors}
In this subsection, we recall some results in \cite{Wang2009} about the upper semicontinuity of random attractors when random disturbance vanishes. Given $ \varepsilon \in ( 0 , 1 ] $ and let $ \Phie $ be a random dynamical system with respect to $ ( \Omega , \mathcal{F} , \bP , ( \theta_{ t } )_{t \in \bR} ) $ which has a random absorbing set $ K _{\varepsilon} = \{ K_{\varepsilon} ( \omega ) \} _{\omega \in \Omega} $ and a random attractor $ \cA_{\varepsilon} = \{ \cA_{\varepsilon} ( \omega ) \} _{\omega \in \Omega} $. Let $ (X, \Vert \cdot \Vert _{X}) $ be a Banach space and $ \Phi $ be a dynamical system defined on $ X $ with the global attractor $ \cA_{0} $, which means that $ \cA_{0} $ is compact and invariant and attracts every bounded subset of $ X $ uniformly.
\begin{definition}
	For $ 0 < \varepsilon \leq 1 $, the family of random attractors $ \{ \cA _{\varepsilon} ( \omega ) \} _{\omega \in \Omega} $ is said to be \textbf{upper semicontinuous} when $ \varepsilon \rightarrow 0^{+} $ if
	$$
		\lim_{\varepsilon \rightarrow 0^{+}} \dist( \cA _{\varepsilon} (\omega) , \cA_{0} ) = 0, \quad \mbox{ for all } \omega \in \Omega.
	$$
\end{definition}

The following proposition is given and proved in \cite{Wang2009}.

\begin{proposition}\label{upperc}
	Suppose that the following conditions hold for $ \bP $-a.e. $ \omega \in \Omega $:
	\begin{enumerate}
		\item [(i)] $$ \lim\limits_{ n \rightarrow \infty } \Phi^{\varepsilon_{n}} ( t , \omega ) x_{n} = \Phi (t) x $$ for all $ t \geq 0 $, provided $ \varepsilon_{n} \rightarrow 0 $ and $ x_{n} \rightarrow x $ in $ X $;
		\item [(ii)] $$ \limsup\limits_{ \varepsilon \rightarrow 0 } \Vert K_{\varepsilon} (\omega) \Vert_{X} \leq c $$
		for some deterministic positive constant $ c $
		where $ \Vert K_{\varepsilon} (\omega) \Vert_{X} = \sup_{x \in K_{\varepsilon} (\omega)} \Vert x \Vert_{X} $;
		\item [(iii)] $$ \bigcup_{0 < \varepsilon \leq 1} \cA_{\varepsilon} ( \omega ) \mbox{ is precompact in } X .$$
	\end{enumerate}
	Then the family of random attractors $ \{ \cA _{\varepsilon} ( \omega ) \} _{\omega \in \Omega} $ is upper semicontinuous as $ \varepsilon \rightarrow 0^{+} $.
\end{proposition}

\subsection{Stochastic Navier-Stokes equations with memory}
In this subsection, we show that there is a continuous random dynamical system generated by the stochastic Navier-Stokes equations with memory in unbounded domains.

First, we set $ \delta_0 = \min \left\{\frac{\nu \lambda_1}{2}, \frac{\delta}{2}\right\} $ and take a constant $ \sigma $ large enough such that
\begin{equation}\label{sigma}
	\sigma > \max \left\{ \frac{c_0 \varepsilon^{2}}{ 2 \delta_0 }, \frac{c_5 \varepsilon^{4}}{ 2 \delta_0 } \right\}.
\end{equation}
where $ c_0 = \frac{2(\hat{c}\tilde{c})^2}{\nu} $ and $ c_5 = \frac{ 3456 ( \hat{c} \tilde{c})^4 }{\nu ^3} $, $ \tilde{c} $ is a constant which will be assigned later.

Next, we consider the probability space $ ( \Omega , \mathcal{F} , \bP ) $ where
$$
	\Omega = \left\{ \omega \in C ( \bR , \bR ) : \omega ( 0 ) = 0 \right\},
$$
$ \mathcal{F} $ is the Borel $ \sigma $-algebra induced by the compact-open topology of $ \Omega $ and $ \bP $ is the corresponding Wiener measure on $ ( \Omega , \mathcal{F} ) $. Then we identify $ W(t) $ with $ \omega(t) $, i.e., $ \omega(t) = W( t , \omega ) $, $ t \in \bR $. Define the time shift by
$$
	\theta_{ t } \omega( \cdot ) = \omega( \cdot + t ) - \omega(t), \quad \omega \in \Omega, \  t \in \bR.
$$
Then $ ( \Omega , \mathcal{F} , \bP , ( \theta_{ t } )_{t \in \bR} ) $ is an ergodic metric dynamical system (see e.g., \cite{Arnold1998}).

Moreover, by applying the Leray orthogonal projection $ \mathcal{P} $ to $\eqref{SNSM}_1$, we have
\begin{equation}\label{SNSM2}
	\pt_t u + \nu Au + \int_{0}^{\infty}g(s) A u (t-s)\d s + B(u,u) = f + \varepsilon h \frac{\d W}{\d t}
\end{equation}
subject to $ u(x, 0) = u_0(x) $ and $ u(x,-s) = \rho(x,s) , s > 0 $ with
\begin{equation}
	\label{initial}
	u_0(x) \in \bH, \quad
	\int_0^s \rho(x, \sigma) \d \sigma \in \cM.
\end{equation}
Here we rewrite $ \mathcal{P} f $ and $ \mathcal{P} h $ as $ f $ and $ h $ respectively and $ f(x) \in \bm{H} $, $ h(x) \in \bm{W} $. Then we introduce the past history variable
$$
	\eta ^t (s) = \int_{0}^{s} u(t-\tau) \d \tau,
$$
which satisfies the differential identity
$$
	\pt_t \eta ^t (s) = - \pt_s \eta^t (s) +  u(t).
$$
For readability, we will suppress $ t $ in the notation of $ \eta $ in the sequel.
Combining the definition of $T$ and integration by parts, one obtains
\begin{equation}\label{SNSM3}
\left\{
	\begin{aligned}
	& \pt_t u + \nu Au + \int_{0}^{\infty} \mu (s) A \eta(s) \d s + B(u,u) = f + \varepsilon h \frac{\d W}{\d t}, \\
	& \pt_t \eta = T \eta +u, \\
	& u(0) = u_0, \  \eta^0 := \eta_0.
	\end{aligned}
\right.
\end{equation}
\begin{remark}
	Though we transfer \eqref{SNSM} to system \eqref{SNSM3}, we can not claim that two systems are equivalent. This is because for the general initial data ($ \eta_0 $ does not depend on $ u_0 $) in the phase space of \eqref{SNSM3}, the solutions are not exactly what of original system \eqref{SNSM} while normally a delayed system has the past history data as $ u(x, t), t >0 $ in more general space. Hence, in our manuscript, we assume a more specific case of past history space by \eqref{initial} and $ \eta_0 = \int_0^s \rho(x, \sigma) \d \sigma $, which is coincide with the phase space of the transformed system, namely,
	\begin{equation}
		\label{initial2}
		(u_0, \eta_0) \in \cH.
	\end{equation}
\end{remark}

To derive a continuous random dynamical system related to Eq. \eqref{SNSM3}, we consider the Ornstein-Uhlenbeck equation \cite{Tang2014} and convert the stochastic equation to a deterministic equation with random parameters $ z $ which satisfies
\begin{equation}\label{OU}
	\d z + \sigma z \d t = \d W.
\end{equation}
It is easy to check that a solution to \eqref{OU} is given by
	$$
		z(t;\omega) = \left( \int_{\infty}^{t} e^{- \nu (t - \tau) }\d W(\tau)\right)(\omega).
	$$
From \cite[Proposition 4.3.3]{Arnold1998}, there exists a tempered function $r(\omega) > 0$ such that
\begin{equation}\label{s1}
	\beta _1 (\theta_t \omega) = \left( \vert z(\theta_t \omega) \vert ^2 + \vert z(\theta_t \omega) \vert ^4\right) \leq r( \theta_{t} \omega ),
\end{equation}
where $ r(\omega )$ satisfies that for $ \bP $-a.e. $ \omega \in \Omega $,
\begin{equation}\label{s2}
	r(\theta_t \omega) \leq e^{\frac{\delta_0}{2} |t|} r(\omega),\quad t\in \bR.
\end{equation}
Then it follows from \eqref{s1} and \eqref{s2} that, for $ \bP $-a.e. $ \omega \in \Omega $,
\begin{equation}\label{s3}
	\beta_1(\theta_t \omega) \leq e^{\frac{\delta_0}{2} |t|} r(\omega), \quad t\in \bR.
\end{equation}

Next, we set $ y(\theta_t \omega) = h(x)z(\theta_t \omega) $ and $ v(t) = u(t) - \varepsilon y(\theta_t \omega) $ where $ u $ is a solution of \eqref{SNSM2}. Since $ h(x) \in \bm{W} $, we denote $ \tilde{c} = \max \left\{ \norm{h(x)}, \norm{A^{\onehalf} h(x)}, \norm{A h(x)} \right\} $ with estimates
\begin{equation}\label{sv}
	\begin{aligned}
		\Vert y(\theta_t \omega) \Vert & \leq \Vert h(x)z(\theta_t \omega) \Vert \leq \tilde{c} \left|z(\theta_t \omega)\right|,\\
		\Vert A^{\frac{1}{2}}y(\theta_t \omega) \Vert  & \leq \Vert A^{\frac{1}{2}} h(x)z(\theta_t \omega) \Vert  \leq \tilde{c} \left|z(\theta_t \omega)\right|,\\
		\Vert A y(\theta_t \omega) \Vert  & \leq \Vert A  h(x)z(\theta_t \omega) \Vert \leq \tilde{c} \left|z(\theta_t \omega)\right|.
	\end{aligned}
\end{equation}
Therefore, we have
\begin{equation}\label{SNSM4}
\left\{
	\begin{aligned}
		&\pt_t v + \nu Av + \int_{0}^{\infty} \mu (s) A \eta (s)\d s + B(v + \varepsilon y(\theta_t \omega) , v + \varepsilon y(\theta_t \omega)) \\
		&\qquad\qquad\qquad\qquad\qquad\qquad = f + \varepsilon \left( \sigma y(\theta_t \omega) - \nu A y(\theta_t \omega)\right), \\
		&\pt_t \eta = T \eta + v + \varepsilon y(\theta_t \omega),\\
		&v(0) = v_0 := u_0 - \varepsilon y(\omega),\ \eta^0 = \eta_0,
	\end{aligned}
\right.
\end{equation}
which is equivalent to \eqref{SNSM2}, providing \eqref{initial2}.

In the following, we focus on system \eqref{SNSM4} combined with the point view of random dynamics in Section \ref{randomdy}. Thus, we denote
\begin{equation}\label{phi}
	\begin{gathered}
		\psi (t,\omega, \psi_0(\omega)) = \tran{v(t, \omega, v_0)}{\eta^t (\omega, \eta_0)} \  \textrm{corresponding to \eqref{SNSM4}}, \\
		\phi (t,\omega, \phi_0(\omega)) = \psi (t,\omega, \psi_0(\omega)) + \tran{\varepsilon y(\theta_t \omega)}{0} \ \textrm{corresponding to \eqref{SNSM3}},
	\end{gathered}
\end{equation}
with
\begin{gather*}
	\psi_{0} = \tran{u_0 - \varepsilon y(\omega)}{\eta_{0}}, \quad
	\phi_{0} = \tran{u_0}{\eta_{0}}, \\
	\Vert \psi \Vert^2_{\cH} = \Vert v \Vert^2 + \Vert \eta \Vert^2_{\mathcal{M}}.
\end{gather*}

\section{Global Well-posedness}\label{wellposd}
We first give the definition of weak solutions:
\begin{definition}[Weak solution]
	Let $ f \in \bm{H} $ and $ \psi_{0} = \tran{v_{0}}{\eta_{0}} \in \cH $, $ \psi = \tran{v}{\eta} $ is a weak solution of problem \eqref{SNSM4} provided that
	\begin{itemize}
		\item[(i)] $ v \in L^{\infty}(0,T; \bm{H}) \cap L^{2}(0,T; \bm{V}), \quad v_{t} \in L^{2}(0,T; \bm{V}'), \quad \eta \in L^{\infty}(0,T; \cM) $;
		\item[(ii)] for all $ \varphi = \tran{w}{\xi} \in \cH $, we have
		\begin{equation}\label{SNSM5}
			\left\{
			\begin{aligned}
				&
				\begin{aligned}
					(\pt_t v, w) + \nu (Av,w) + \int_{0}^{\infty} \mu (s) (A \eta (s), w)\d s + b(v + \varepsilon y(\theta_t \omega) , v + \varepsilon y(\theta_t \omega), w) & \\
					= ( f + \varepsilon \left( \sigma y(\theta_t \omega) - \nu A y(\theta_t \omega)\right) , w)&,
				\end{aligned} \\
				& \innerM{\pt_t \eta}{\xi}= \innerM{T \eta}{\xi} + \innerM{v + \varepsilon y(\theta_t \omega)}{\xi}.
			\end{aligned}
			\right.
		\end{equation}
	\end{itemize}
\end{definition}
Then, the global well-posedness of \eqref{SNSM4} is stated in the following theorem.
\begin{theorem}\label{wellp}
	Assume that $ T_{0} \in \bR^{+} $ and $ \psi_{0} \in \cH $. For $ \bP $-$ a.e. $ $ \omega \in \Omega $, there exists a unique solution $ \psi $ of \eqref{SNSM4} on the interval $ [T_{0}, \infty) $ satisfying
	$$
		v \in L^{\infty}(T_{0},\infty; \bm{H}) \cap L^{2}(T_{0},\infty; \bm{V}), \quad \eta \in L^{\infty}(T_{0},\infty; \cM)
	$$
	and $ v_{t} $ is uniformly bounded in $ L^{2}(T_{0},\infty; \bm{V}') $. Moreover, the solution continuously depends on
	the initial data.
\end{theorem}
\begin{proof}
	We divide the proof of Theorem \ref{wellp} into three parts: existence, uniqueness and dependence.
	
	\textbf{Part 1: Existence}. First, we use the standard Faedo-Galerkin procedure to show the existence of weak solution to \eqref{SNSM4}.\\
	\uuline{\textit{Step 1}. The approximate system.} Let $ \left\{ w_{j} \right\}_{j=1}^{\infty} $ be the normalized eigenfunction basis of the Stokes operator $ A $ and $ \left\{ \xi_{j} \right\}_{j=1}^{\infty} $ be the orthonormal basis of $ \cM $ with all $ \xi_{j} \in C^{\infty}_{0}(\bR^{+}, \bV) $ where $ C^{\infty}_{0}(\bR^{+}, \bV) $ is a compactly supported infinitely differentiable function space respect to $ s $. For any integer $ n $, we denote by $ P_{n} $ and $ Q_{n} $ the projections onto the subspaces
	$$
		H_{n} := P_{n} H = \mbox{span}
		\left\{ w_{1}, w_{2}, \dots, w_{n} \right\}
		\subset \bH
	$$
	and
	$$
		\cM_{n} := Q_{n} \cM = \mbox{span}
		\left\{ \xi_{1}, \xi_{2}, \dots, \xi_{n} \right\}
		\subset \cM,
	$$
	respectively. Then we define the approximate solutions $ \psi_{n} := \tran{\vn}{\etan} $ as
	$$
		\vn(t) = \sum_{k=1}^{n} a^{n}_{k}(t) w_{k},
		\quad
		\etan(s) = \sum_{k=1}^{n} b^{n}_{k}(t) \xi_{k}(s),
	$$
	with respect to unknown smooth variables $ \tran{a_k^n}{b_k^n}(t) $. Hence we have
	$$
		A \vn(t) = \sum_{k=1}^{n} \lambda_{k} a^{n}_{k}(t) w_{k},
		\quad
		A \etan(s) = \sum_{k=1}^{n} \lambda_{k} b^{n}_{k}(t) \xi_{k}(s).
	$$
	Therefore, the $ n $th-order Galerkin approximate system is
	\begin{equation}\label{APPS}
		\left\{
		\begin{aligned}
			&
			\begin{aligned}
			(\pt_t\vn, w) + \nu (A\vn,w) + \int_{0}^{\infty} \mu (s) (A \etan (s), w)\d s  & \\
			+ b(\vn + \varepsilon y , \vn + \varepsilon y, w)
			= ( \tilde{f} , w), &
			\end{aligned} & \mbox{ for all } w \in H_{n} ,\\
			& \innerM{\pt_t \etan}{\xi}= \innerM{T \etan}{\xi} + \innerM{\vn + \varepsilon y}{\xi}, & \mbox{ for all } \xi \in \cM_{n},\\
			& \vn(0) = P_{n}v_{0},  & a.e. \mbox{ in } \cO, \\
			& \etan(0) = Q_{n}\eta^{0}, & a.e. \mbox{ in } \cO \times \bR^{+}.
		\end{aligned}
		\right.
	\end{equation}
	where we denote $ \tilde{f} $ by $ \tilde{f} = P_{n} \left[ f + \varepsilon \left( \sigma y - \nu A y \right) \right] $.

	Let $ \varphi := \tran{w}{\xi} $. Taking $\varphi = \tran{w_{k}}{\xi_{k}}\ k = 1, 2, \dots, n $, we deduce that
	\begin{equation}\label{ODE}
		\left\{
			\begin{aligned}
				& \frac{\d}{\d t} X = G(t,X(t)) \\
				& X(0) = \tran{\inner{v_{0}}{a_{k}^{n}}}{\innerM{\eta^{0}}{b_{k}^{n}}}
			\end{aligned}
		\right.
	\end{equation}
	where $ X(t) = \tran{a_k^n}{b_k^n} $, $ G = G_k $, $ k = 1, 2, \dots, n $ and
	\begin{equation*}
		G_k =
		\left(
			\begin{gathered}
				- \nu \lambda_{k} a^{n}_{k}
				- \sum_{j=1}^{n} b^{n}_{j} \innerM{\xi_{j}}{w_{k}}
				- b\bigg(\sum_{i=1}^{n} a^{n}_{i} w_{i} + \varepsilon y , \sum_{j=1}^{n} a^{n}_{j} w_{j} + \varepsilon y, w_{k}\bigg)
				+ ( \tilde{f} , w_{k}) \\
				\sum_{j=1}^{n} a^{n}_{j} \innerM{w_{j}}{\xi_{k}}
				- \sum_{j=1}^{n} b^{n}_{j} \innerM{\xi'_{j}}{\xi_{k}}
				+ \innerM{\varepsilon y}{\xi_{k}},
			\end{gathered}
		\right).
	\end{equation*}
	Since $ w_k \in H_n \subset \bH$, $ \xi_k \in \cM_n \subset \cM $, we notice that the coefficients in $ G_k $ satisfy
	\begin{align*}
		& \innerM{\xi_{j}}{w_{k}}
		= \alpha_{j,k}, \\
		& b\bigg(\sum_{i=1}^{n} a^{n}_{i} w_{i}, \sum_{j=1}^{n} a^{n}_{j} w_{j}, w_{k}\bigg)
		= \sum_{i=1}^{n} \sum_{j=1}^{n} \beta_{ij,k} a^{n}_{i} a^{n}_{j}, \\
		& b\bigg(\sum_{i=1}^{n} a^{n}_{i} w_{i}, \varepsilon y, w_{k}\bigg)
		= \sum_{i=1}^{n} \gamma_{i,k} a^{n}_{i}, \quad b\bigg(\varepsilon y , \sum_{j=1}^{n} a^{n}_{j} w_{j}, w_{k}\bigg)
		= \sum_{j=1}^{n} \zeta_{j,k} a^{n}_{j}, \\
		& b(\varepsilon y, \varepsilon y, w_k) + ( \tilde{f} , w_{k})
		= C_k, \quad \innerM{\xi'_{j}}{\xi_{k}} = \iota_{j,k}, \quad \innerM{\varepsilon y}{\xi_{k}} = D_k,
	\end{align*}
	where $ \alpha_{j,k} $, $ \beta_{ij,k} $, $ \gamma_{i,k} $, $ C_k $, $ \iota_{j,k} $, $ D_k $ $ \in \bR $ are finite constants. Then the continuity of $ a_k^n $ and $ b_k^n $ implies that $ G(t,X(t)) $ is continuous about $ t $. Moreover, for every $ R > 0 $, there is a constant $ L > 0 $ depending only on $ R $, such that for $ t \in (0,T) $ and $ \abs{X}, \abs{Y} \leq R  $, we have
	\begin{equation*}
		\abs{G(t,X) - G(t,Y)} \leq L \abs{X - Y},
	\end{equation*}
	which means the local Lipschitz condition for $ G(t, \cdot) $. Thanks to the local existence theory of ordinary differential equations, we know that there exists a solution to the problem \eqref{ODE} which means that solutions $ \tran{\vn}{\etan} $ to the approximate problem \eqref{APPS} exist.\\
	\uuline{\textit{Step 2}. Uniform \textit{a priori} estimates.}
	Let us take $ \varphi = \tran{\vn}{\etan} $ in problem \eqref{APPS}, then we have
	\begin{equation}\label{Energy1}
		\begin{aligned}
			& \frac{1}{2} \frac{\d}{\d t} \left( \norm{ \vn }^{2} + \normM{ \etan }^{2} \right) + \nu \norm{ A^{\frac{1}{2}} \vn }^{2} \\
			& \quad = \frac{1}{2} \int_{0}^{\infty} \mu'(s) \norm{ A^{\frac{1}{2}} \etan }^{2} \d s + \varepsilon \innerM{\etan}{y(\theta_t \omega)} \\
			& \qquad -  b(\vn + \varepsilon y(\theta_t \omega) , \vn + \varepsilon y(\theta_t \omega), \vn)\\
			& \qquad + (\tilde{f} , \vn) \\
			& \quad =: I_1 + I_2 + I_3 + I_4.
		\end{aligned}
	\end{equation}
	It is clear from the Poincar\'{e} inequality that
	\begin{equation}\label{e0}
	 	\nu \norm{ A^{\frac{1}{2}} \vn }^{2} \geq \frac{\nu}{2} \norm{ \vn }^{2}_1
	 	+ \frac{\nu \lambda_1}{2} \norm{ \vn }^{2}.
	\end{equation}
	
	Now we are ready to estimate $ I_i, i=1,2,3,4 $ respectively.
	For $I_1$, from the Dafermos condition \eqref{mu}, we get
	\begin{equation}\label{e1}
	 	I_1 \leq -\frac{\delta}{2} \int_{0}^{\infty} \mu (s) \norm{ A^{\frac{1}{2}} \etan }^{2} \d s = -\frac{\delta}{2} \normM{ \etan }^{2}.
	\end{equation}
	By utilizing the H\"{o}lder's inequality, the Young's inequality with $ \epsilon $ and \eqref{sv}, one obtains that
	\begin{equation}\label{e2}
	 	\begin{aligned}
			I_2
			& = \varepsilon \int_{0}^{\infty} \mu (s) 		(A^{\frac{1}{2}}\etan,A^{\frac{1}{2}}y(\theta_t\omega))\d s\\
			&\leq \frac{\delta}{4}\int_{0}^{\infty} \mu (s) \norm{ A^{\frac{1}{2}} \etan }^{2}\d s
			+ \frac{\varepsilon^2}{\delta} \int_{0}^{\infty} \mu (s) \Vert A^{\frac{1}{2}} y(\theta_t\omega) \Vert ^2\d s\\
			&\leq \frac{\delta}{4} \normM{ \etan }^{2}
			+ \frac{c \varepsilon^2 \kappa}{\delta} \beta_{1}(\theta_{ t } \omega).
		\end{aligned}
	\end{equation}
	Since $b(u,v,v) = 0$, it follows from the H\"{o}lder's inequality, the Young's inequality with $ \epsilon $, Lemma \ref{b} and \eqref{sv} that
	\begin{equation}\label{e3}
		\begin{aligned}
			\vert I_3 \vert
			& \leq \vert b(\vn , \varepsilon y(\theta_t \omega) , \vn) \vert
			+ \vert b(\varepsilon y(\theta_t \omega) , \varepsilon y(\theta_t \omega) , \vn) \vert \\
			& \leq \hat{c} \varepsilon \norm{\vn} \norm{A^{\onehalf} y(\theta_{ t } \omega)} \norm{A^\onehalf \vn} \\
			& \quad + \hat{c} \varepsilon^{2} \norm{y(\theta_t \omega)}^\onehalf \norm{A^\onehalf y(\theta_t \omega)} \norm{A y(\theta_t \omega)}^\onehalf \norm{\vn}\\
			& \leq \frac{\hat{c}^2 \tilde{c}^2 \varepsilon^2}{\nu} \beta_1(\theta_{ t }\omega) \norm{\vn}^{2}
			+ \frac{\nu}{4} \norm{A^\onehalf \vn}^2 + \frac{\hat{c}^2 \tilde{c}^4 \varepsilon^4}{\nu \lambda_1} \beta_1(\theta_{ t } \omega) + \frac{\nu \lambda_{1}}{8} \norm{\vn}^2.
		\end{aligned}
	\end{equation}
	As $ f \in \bm{H} $, it can be deduced from the Young's inequality with $ \epsilon $ and \eqref{sv} that
	\begin{equation}\label{e4}
		\begin{aligned}
			I_4
			& \leq \frac{\nu \lambda_1}{8} \Vert \vn \Vert ^2 + \frac{2}{\nu \lambda_1} \left(\Vert f \Vert ^2 + 2c \varepsilon^{2} \left(\Vert y(\theta_t \omega) \Vert ^2 + \Vert A y(\theta_t \omega) \Vert ^2\right) \right) \\
			& \leq \frac{\nu \lambda_1}{8} \Vert \vn \Vert ^2 + \frac{c}{\nu \lambda_1} \left( 1 + \varepsilon^{2} \beta_1(\theta_{ t } \omega) \right).
		\end{aligned}
	\end{equation}
	Combining \eqref{Energy1}--\eqref{e4}, we have
	\begin{equation}\label{E1}
		\begin{aligned}
			& \quad \frac{\d}{\d t} \Vert \psin (t,\omega,\psi_0(\omega))\Vert^2_{\cH} + \frac{\nu}{2} \norm{ \vn (t,\omega,v_0(\omega)) }^{2} _1 \\
			& \leq \left( - \delta_0 + c_0 \varepsilon^{2} \beta_{1}(\theta_{ t } \omega)\right) \Vert \psin (t,\omega,\psi_0(\omega))\Vert^2_{\cH}
			 + c \left(1 + \varepsilon^{2} \beta _1 (\theta_t \omega)\right).
		\end{aligned}
	\end{equation}
	Applying the Gronwall's inequality, one obtains
	\begin{equation}\label{eH1}
	 	\begin{aligned}
	 		& \Vert \psin (t,\omega,\psi_0(\omega))\Vert^2_{\cH} \\
	 		& \quad	+ \frac{\nu}{2} \int_{0}^{t} e ^{ \delta_0 (s-t) - c_0 \varepsilon^{2} \int_{t}^{s} \beta_{1}(\theta_{ \tau } \omega) \d \tau } \Vert \vn(s,\omega,v_0(\omega)) \Vert^2 _1 \d s \\
			& \leq e ^{-\delta_0 t + c_0 \varepsilon^{2} \int_{0}^{t} \beta_{1}(\theta_{ \tau } \omega) \d \tau } \Vert \psi_0 (\omega)\Vert^2_{\cH}\\
			& \quad	+ c \int_{0}^{t} e ^{ \delta_0 (s-t) - c_0 \varepsilon^{2} \int_{t}^{s} \beta_{1}(\theta_{ \tau } \omega) \d \tau } \left(1 + \varepsilon^{2} \beta _1 (\theta_s \omega)\right)\d s.
		\end{aligned}
	\end{equation}
	In \eqref{eH1}, we replace $\omega$ by $\theta_{ - t} \omega$, then
	\begin{equation}\label{eH2}
		\begin{aligned}
			& \Vert \psin (t,\theta_{ - t} \omega,\psi_0(\theta_{ - t} \omega))\Vert^2_{\cH} \\
	 		& \quad + \frac{\nu}{2} \int_{0}^{t} e ^{ \delta_0 (s-t) - c_0 \varepsilon^{2} \int_{t}^{s} \beta_{1}(\theta_{ \tau - t } \omega) \d \tau } \Vert \vn( s , \theta_{ - t } \omega , v_{0}(\theta_{ - t } \omega) ) \Vert^2 _1 \d s \\
			& \quad \leq e ^{-\delta_0 t + c_0 \varepsilon^{2} \int_{0}^{t} \beta_{1}(\theta_{ \tau - t } \omega) \d \tau } \Vert \psi_0 (\theta_{ - t} \omega) \Vert^2_{\cH} \\
			& \qquad + c \int_{0}^{t} e ^{ \delta_0 (s-t) - c_0 \varepsilon^{2} \int_{t}^{s} \beta_{1}(\theta_{ \tau - t } \omega) \d \tau } \left(1 + \varepsilon^{2} \beta _1 (\theta_{s-t} \omega)\right)\d s\\
			& \quad = e ^{-\delta_0 t + c_0 \varepsilon^{2} \int_{-t}^{0} \beta_{1}(\theta_{ \tau } \omega) \d \tau } 	\Vert \psi_0 (\theta_{ - t} \omega) \Vert^2_{\cH} \\
			& \qquad + c \int_{-t}^{0} e ^{ \delta_0 s + c_0 \varepsilon^{2} \int_{s}^{0} \beta_{1}(\theta_{ \tau } \omega) \d \tau } \left(1 + \varepsilon^{2} \beta _1 (\theta_s \omega)\right)\d s.
		\end{aligned}
	\end{equation}
	
	Since $\beta_{1}(\theta_{ \tau } \omega)$ is stationary and ergodic, it follows from the ergodic theorem in \cite{CDF1997} that
	$$
	 	\lim\limits_{t \rightarrow + \infty} \frac{1}{t} \int_{-t}^{0} \beta_{1}(\theta_{ \tau } \omega) \d \tau = \mathbb{E} \left( \beta_{1}(\omega) \right) \leq \frac{1}{4 \sigma}.
	$$
	Hence, there exists a $T_1 > 0$ such that for all $t \geq T_1$,
	\begin{equation}\label{zz}
		\int_{-t}^{0} \beta_{1}(\theta_{ \tau } \omega) \d \tau \leq \frac{1}{4 \sigma} t \leq \frac{\delta_0}{2 c_0 \varepsilon^{2}} t.
	\end{equation}
	Thus for all $t \geq T_1$,
	\begin{equation}\label{eH3}
	 	\begin{aligned}
	 		& \quad \Vert \psin (t,\theta_{ - t} \omega,\psi_0(\theta_{ - t} \omega))\Vert^2_{\cH}
	 		+ \frac{\nu}{2} \int_{0}^{t} \Vert \vn( s , \theta_{ - t } \omega , v_{0}(\theta_{ - t } \omega) ) \Vert^2 _1 \d s \\
	 		& \leq e ^{-\frac{\delta_0}{2} t } \Vert \psi_0 (\theta_{ - t} \omega)\Vert^2_{\cH}
	 		+ c \int_{-t}^{0} e ^{ \delta_0 s + c_0 \varepsilon^{2} \int_{s}^{0} \beta_{1}(\theta_{ \tau } \omega) \d \tau } \left(1 + \varepsilon^{2} \beta _1 (\theta_s \omega) \right) \d s.
		\end{aligned}
	\end{equation}
	Note that $\psi_0 (\theta_{ - t} \omega) \in B(\theta_{ - t} \omega) \subset \cD$, we see that there is a $T_2 = T_2(B, \omega) > 0$, independent of $\varepsilon$, such that for all $t \geq T_2$,
	$$
		e ^{-\frac{\delta_0}{2} t } \Vert \psi_0 (\theta_{ - t} \omega)\Vert^2_{\cH} \leq 1.
	$$
	Since $ -t < s < \tau < 0 $, it follows from \eqref{zz} that for all $ t \geq T_{1} $,
	\begin{equation}\label{st}
	 	\int_{s}^{0} \beta_{1}(\theta_{ \tau } \omega) \d \tau
	 	\leq \frac{\delta_0}{2 c_0 \varepsilon^{2}} t.
	\end{equation}
	
	We now denote $ r _{1} ^{ \varepsilon } ( \omega ) $ by
	\begin{equation*}
	 	r _{1} ^{ \varepsilon } ( \omega )  = \int_{-\infty}^{0} e ^{ \delta_0 s + c_0 \varepsilon^{2} \int_{s}^{0} \beta_{1}(\theta_{ \tau } \omega) \d \tau } \left(1 + \varepsilon^{2} \beta _1 (\theta_s \omega) \right) \d s.
	\end{equation*}
	It follows from \eqref{s3} and \eqref{st} that
	\begin{align*}
		r _{1} ^{ \varepsilon } ( \theta_{ - t } \omega )
		& = \int_{-\infty}^{0} e ^{ \delta_0 s + c_0 \varepsilon^{2} \int_{s}^{0} \beta_{1}(\theta_{ \tau - t } \omega) \d \tau } \left(1 + \varepsilon^{2} \beta _1 (\theta_{s - t} \omega) \right) \d s \\
		& = \int_{-\infty}^{0} e ^{ \delta_0 s + c_0 \varepsilon^{2} \int_{s - t}^{- t} \beta_{1}(\theta_{ \tau } \omega) \d \tau } \left(1 + \varepsilon^{2} \beta _1 (\theta_{s - t} \omega) \right) \d s \\
		& \leq \int_{-\infty}^{0} e ^{ \frac{\delta_0}{2} s + c_0 \varepsilon^{2} \int_{s - t}^{- t} \beta_{1}(\theta_{ \tau } \omega) \d \tau } \left(1 + \varepsilon^{2} \beta _1 (\theta_{s - t} \omega) \right) \d s \\
		& = \int_{-\infty}^{-t} e ^{ \frac{\delta_0}{2} (s + t) + c_0 \varepsilon^{2} \int_{s}^{- t} \beta_{1}(\theta_{ \tau } \omega) \d \tau } \left(1 + \varepsilon^{2} \beta _1 (\theta_{s} \omega) \right) \d s \\
		& \leq \int_{-\infty}^{-t} e ^{ \frac{\delta_0}{2} (s + t) + c_0 \varepsilon^{2} \int_{s}^{0} \beta_{1}(\theta_{ \tau } \omega) \d \tau } \left(1 + \varepsilon^{2} \beta _1 (\theta_{s} \omega) \right) \d s \\
		& \leq e ^{ \delta_0 t } \int_{-\infty}^{-t} e ^{ \frac{\delta_0}{2} s } \left(1 + \varepsilon^{2} e^{- \frac{\delta_0}{2} s} r( \omega ) \right) \d s \\
		& = \frac{2}{\delta_0} e^{ \frac{\delta_0}{2} t }
		+ \varepsilon^{2} r ( \omega ) te ^{ \delta_0 t }.
	\end{align*}
	Then we get
	\begin{equation*}
		e ^{ - \frac{3 \delta_0}{2} t } r _{1} ^{ \varepsilon } ( \theta_{ - t } \omega )
		= \frac{2}{\delta_0} e^{ - \delta_0 t }
		+ \varepsilon^{2} r ( \omega ) t e ^{ - \frac{\delta_0}{2} t }
		\rightarrow 0, \mbox{ as } t \rightarrow \infty,
	\end{equation*}
	which means that $ r _{1} ^{ \varepsilon } ( \omega )  $ is a tempered function.
	Then for all $t \geq T_3 (B, \omega) = \max \{T_1, T_2\}$,
	\begin{equation}\label{conclusion1}
	 	\begin{aligned}
	 		& \quad \Vert \psin (t,\theta_{ - t} \omega,\psi_0(\theta_{ - t} \omega))\Vert^2_{\cH}
	 		+ \frac{\nu}{2} \int_{0}^{t} \Vert \vn( s , \theta_{ - t } \omega , v_{0}(\theta_{ - t } \omega) ) \Vert^2 _1 \d s
	 		\leq c_1 (r _{1} ^{ \varepsilon } ( \omega )  + 1).
	 	\end{aligned}
	\end{equation}
	where $ c_1 $ is a constant independent of $ \varepsilon $.
	Therefore, we deduce that for $ T > T_{3} $,
	\begin{align}\label{unifb1}
		\vn & \mbox{ is uniformly bounded in } L^{\infty} (T_{3},\infty; \bm{H}) \cap L^{2} (T_{3},\infty; \bm{V}),\\\label{unifb2}
		\etan & \mbox{ is uniformly bounded in } L^{\infty} (T_{3},\infty; \cM).
	\end{align}	
	
	Next, we estimate $ \pt_t \vn $. Taking any $ w \in \bV $ in equation \eqref{APPS}$ _{1} $, we have
	\begin{equation*}
		(\pt_t\vn, w) =
		- \nu (A\vn,w)
		- \int_{0}^{\infty} \mu (s) (A \etan (s), w)\d s
		- b(\vn + \varepsilon y , \vn + \varepsilon y, w)
		+ ( \tilde{f} , w),
	\end{equation*}
	Since
	\begin{gather*}
		\abs{(A\vn,w)} \leq \norm{\nabla \vn} \norm{w}_{\bV}; \\
		\abs{\int_{0}^{\infty} \mu (s) (A \etan (s), w)\d s}
		\leq \kappa \normM{\etan} \norm{w}_{\bV}; \\
		\begin{aligned}
			\abs{b(\vn + \varepsilon y , \vn + \varepsilon y, w)}
			& \leq c \beta_1^{\onehalf}(\theta_{ t } \omega) \left( \varepsilon^{2} + \varepsilon \norm{\nabla \vn} \right) \norm{w}_{\bV} \\
			& \quad + \norm{ \vn } \norm{\nabla \vn} \norm{w}_{\bV};
		\end{aligned} \\
		( \tilde{f} , w) \leq c ( 1 + \varepsilon \beta_1^{\onehalf} (\theta_{ t }\omega) ) \norm{w}_{\bV},
	\end{gather*}
	we obtain
	$$
		\begin{aligned}
			\norm{\pt_t \vn}_{\bV'}
			& \leq \norm{\nabla \vn} + \kappa \normM{\etan}
			+ c \beta_1^{\onehalf} (\theta_{ t } \omega) \left( \varepsilon^{2} + \varepsilon \norm{\nabla \vn} \right) \\
			& \quad + \norm{ \vn } \norm{\nabla \vn}
			+ c ( 1 + \beta_1^{\onehalf} (\theta_{ t }\omega) ) .
		\end{aligned}
	$$
	It follows from \eqref{unifb1} that for $ T > T_{3} $,
	\begin{equation*}
		\begin{aligned}
			\int_{T_{3}}^{T} \norm{\pt_t \vn}_{\bV'} ^{2} \d s
			& \leq c \int_{T_{3}}^{T} \norm{\nabla \vn} ^{2} \d s
				+ c \int_{T_{3}}^{T} \normM{\etan} ^{2} \d s \\
				& \quad + c \int_{T_{3}}^{T} \beta_1 (\theta_{ s } \omega) \left( \varepsilon^{4} + \varepsilon^{2} \norm{\nabla \vn}^{2} \right) \d s \\
				& \quad + c \int_{T_{3}}^{T} \norm{ \vn }^{2} \norm{\nabla \vn}^{2} \d s
				+ c \int_{T_{3}}^{T} ( 1 + \varepsilon^{2} \beta_1 (\theta_{ s }\omega) ) \d s\\
			& \leq c_{T} (r _{1} ^{ \varepsilon } ( \omega )  + 1).
		\end{aligned}
	\end{equation*}
	where $ c_{T} $ is a finite constant depending on $ T $.
	Then,
	\begin{equation}\label{unifb3}
		\pt_{t} \vn \mbox{ is uniformly bounded in } L^{2} (T_{3},\infty; \bV').
	\end{equation}
	\\
	\uuline{\textit{Step 3}. Passing to the limits.} It follows from \eqref{unifb1}, \eqref{unifb2} and \eqref{unifb3} that
	\begin{align}
		\label{Linfity}
		\vn \rightharpoonup v & \mbox{ weakly* in } L^{\infty} (T_{3},\infty; \bm{H}); \\
		\vn \rightharpoonup v & \mbox{ weakly in } L^{2} (T_{3},\infty; \bm{V}); \\
		\etan \rightharpoonup \eta^{t} & \mbox{ weakly* in } L^{\infty} (T_{3},\infty; \cM); \\
		\pt_{t} \vn \rightharpoonup \pt_{t} v & \mbox{ weakly in } L^{2} (T_{3},\infty; \bV').
	\end{align}
%
	Following the procedure in Brze\'{z}niak and Li \cite{BL2006}, who constructed a bounded ball and applied the standard diagonal procedure to discuss the compactness in unbounded domains, one gets
	\begin{equation}
		\label{L2}
		\vn \rightarrow v \mbox{ strongly in } L^{2} (T_{3},\infty; \bH),
	\end{equation}
	with $ v \in C(T_3, \infty; \bH) $. Then let us integrate \eqref{APPS} with respect to time over $ (T_3,T) $, $ T > T_3 $, for each term of which, we pass to the limit, i.e., for all $ w \in \bV $ and $ \xi \in \cM^1 $,
	\begin{align*}
		& \int_{T_3}^T (\pt_t \vn, w) \rightarrow \int_{T_3}^T (\pt_t v, w),
		\\
		& \int_{T_3}^T \nu (A\vn -A v,w)
		= \int_{T_3}^T \nu (A^\onehalf \vn - A^\onehalf  v, A^\onehalf w)
		\rightarrow 0, \\
		& \int_{T_3}^T \int_{0}^{\infty} \mu (s) (A \etan (s) - A \eta(s), w) \\
		& \quad = \int_{T_3}^T \int_{0}^{\infty} \mu (s) (A^\onehalf \etan (s) - A^\onehalf \eta(s), A^\onehalf w)
		\rightarrow 0, \\
		& \int_{T_3}^T b(\vn - v, \varepsilon y, w)
		+ b(\varepsilon y , \vn - v, w) \\
		& \quad \leq C \int_{T_3}^T \norm{\vn - v} \norm{\nabla (\varepsilon yw)} \\
		& \quad \leq C \left(\int_{T_3}^T \norm{\vn - v}^2\right)^\onehalf
		\left(\int_{T_3}^T \norm{\nabla (\varepsilon yw)} ^2\right)^\onehalf
		\rightarrow 0,
	\end{align*}
	and
	\begin{align*}
		&\int_{T_3}^T b(\vn , \vn , w) - b(v , v , w) \\
		& \quad = \int_{T_3}^T b(\vn , \vn - v, w) + b(\vn - v , v , w) \\
		& \quad = \int_{T_3}^T - b(\vn , w, \vn - v) + b(\vn - v , v , w) \\
		& \quad \leq C \int_{T_3}^T \norm{\vn - v} \left(\norm{\vn \nabla w}  + \norm{\nabla v w} \right) \\
		& \quad \leq C \left(\int_{T_3}^T \norm{\vn - v}^2\right)^\onehalf \left(\int_{T_3}^T \left(\norm{\vn \nabla w}  + \norm{\nabla v w} \right)^2 \right)^\onehalf
		\rightarrow 0,
	\end{align*}
	where we used \eqref{Linfity}--\eqref{L2}. Additionally, thanks to the regularity of limit functions $ v $ and $ \eta $, we have
	\begin{equation*}
		\eta^t(s) =
		\left\{
			\begin{aligned}
				& \int_{t-s}^t \left( v(\tau) + \varepsilon y(\theta_{ \tau }\omega )\right) \d \tau, \quad & 0 < s \leq t - T_3, \\
				& \eta_{0}(s - t + T_3) + \int_{T_{3}}^t \left( v(\tau) + \varepsilon y(\theta_{ \tau }\omega ) \right) \d \tau, \quad & s > t - T_3.
			\end{aligned}
		\right.
	\end{equation*}
	Therefore, $ (\vn, \etan) $ converges to the solution of \eqref{SNSM4}. Readers can find similar arguments in \cite[Section 5.1]{BL2006} and \cite[Section 4.4]{GGP1999}.
	
	\textbf{Part 2: Uniqueness}. Suppose that $ \psi_{1} = \tran{v_{1}}{\eta_{1}} $ and $ \psi_{2} = \tran{v_{2}}{\eta_{2}} $ are two solutions to \eqref{SNSM4} with the same initial data. Also we define $ \hpsi := \tran{w}{\xi} = \psi_{1} - \psi_{2} $. Then we have the following system
	\begin{equation}\label{UNIQ}
		\left\{
		\begin{aligned}
			&\pt_t w + \nu Aw + \int_{0}^{\infty} \mu (s) A \xi (s)\d s + B(w , v_{1} + \varepsilon y) + B(v_{2} + \varepsilon y , w) = 0, \\
			&\pt_t \xi = T \xi + w,\\
			&w(0) = 0,\ \xi^0 = 0.
		\end{aligned}
		\right.
	\end{equation}
	
	Taking $ \cH $ inner product with \eqref{UNIQ} by $ \hpsi $, we obtain
	\begin{equation}\label{UNIQ2}
		\frac{1}{2} \frac{\d}{\d t} \left( \norm{w}^{2} + \normM{\xi}^2 \right) + \nu \norm{A^{\onehalf} w}^{2}
		= \onehalf \int_{0}^{\infty} \mu'(s) \norm{A^{\onehalf} \xi(s)}^{2} \d s - b(w,v_{1} + \varepsilon y, w).
	\end{equation}
	As the similar procedure above, we get
	\begin{gather*}
		\nu \norm{A^{\onehalf} w}^{2} \geq \frac{\nu}{2} \norm{A^{\onehalf} w}^{2} + \frac{\nu \lambda_1}{2} \norm{w}^{2}; \\
		\onehalf \int_{0}^{\infty} \mu'(s) \norm{A^{\onehalf} \xi(s)}^{2} \d s \leq \frac{\delta}{2} \int_{0}^{\infty} \mu(s) \norm{A^{\onehalf} \xi(s)}^{2} \d s; \\
		\abs{b(w,\varepsilon y ,w)} \leq \frac{c_{0} \varepsilon^{2}}{\nu} \beta_1(\theta_{ t } \omega) \norm{w}^{2} + \frac{\nu}{8} \norm{A^{\onehalf} w}^{2}; \\
		\abs{b(w, v_{1} ,w)} \leq \frac{2 \hat{c}^{2}}{\nu} \norm{\nabla v_{1}}^{2}  \norm{w}^{2} + \frac{\nu}{8} \norm{A^{\onehalf} w}^{2}.
	\end{gather*}
	Hence
	\begin{equation}\label{UNIQ3}
		\frac{\d}{\d t} \norm{\hat{\psi}}^{2}_{\cH}  + \frac{\nu}{2} \norm{A^{\onehalf} w}^{2}
		\leq \left(-2\delta_{0} + \frac{2 c_{0} \varepsilon^{2}}{\nu} \beta_{1}(\theta_{ t } \omega) + \frac{4 \hat{c}^{2}}{\nu} \norm{\nabla v_{1}}^{2} \right) \norm{\hat{\psi}}^{2}_{\cH} .
	\end{equation}
	It can be deduced from \eqref{unifb1} that
	\begin{equation}
		\norm{\hat{\psi}( t, \theta_{ - t } \omega, \hat{\psi}_{0} )}^{2}_{\cH}
		\leq e ^{- \delta_{0} t + c (r_{1}^{\varepsilon}(\omega) + 1)} \norm{ \hat{\psi}_{0} }^{2}_{\cH}  = 0, \mbox{ as } \hat{\psi}_{0} = 0,
	\end{equation}
	which means the solution is unique.
	
	\textbf{Part 3: Dependence}. We assume that $ \psi_{1} = \tran{v_{1}}{\eta_{1}} $ and $ \psi_{2} = \tran{v_{2}}{\eta_{2}} $ are two solutions to \eqref{SNSM4} subjected to different initial datum $ \psi_{01} $ and $ \psi_{02} $, respectively. Reasoning as in uniqueness, we obtain the dependence of initial datum immediately.
\end{proof}

\begin{remark}
	Gao and Sun \cite{Gao2012} showed the well-posedness of 3D stochastic Navier-Stokes-Voigt equations in bounded domains. If we consider 3D stochastic Navier-Stokes-Voigt equations in some unbounded domains, we can obtain the well-posedness thanks to the Voigt term $ - \alpha \Delta (\pt_t u) $ by means of the arguments in \cite{AT2013,Gao2012}. However, when we take \eqref{SNSM} into account in three dimensional case, in which the memory effects is weaker than $ - \alpha \Delta (\pt_t u) $, the uniqueness is lost since the uniqueness of classical 3D Navier-Stokes equations is still open. Fortunately, two dimensional incompressible Navier-Stokes equations are well-posedness, so in our work, we investigate system \eqref{SNSM} in two dimension successfully.
\end{remark}
\section{Random attractors}\label{attractor}
In this section, we aim to establish uniform estimates for $ \psi $ and $ \phi $ with respect to the small parameter $ \varepsilon $, including long-time \textit{a priori} estimates and far-field estimates. We decompose the solution into two parts and get the higher order estimate. By means of a constructed compact subspace, we prove the compactness of solution and show the existence of random attractor.

\subsection{Uniform estimates for solutions}
Now we define a mapping $ \Phi: \bR^+ \times \Omega \times \cH \rightarrow \cH $ by
\begin{equation*}
	\Phi(t, \omega)\phi_0 = \phi(t, \omega, \phi_0).
\end{equation*}
From the definition in \ref{RDS}, we observe that $ \Phi $ is a continuous random dynamical system related to \eqref{SNSM3}.

In this subsection, we impose the uniform estimates for solutions in $ \cH $ and get the absorbing set.
\begin{lemma}\label{lem1}
	For every $ B(\omega)_{\omega \in \Omega} \in \cD $ and for $ \bP $-$ a.e. $ $ \omega \in \Omega $, there exist $ T_{3} = T_{3}(B , \omega) > 0 $ and a tempered function $ r _{1} ^{ \varepsilon } ( \omega ) $, such that for all $ \psi_0 ( \theta_{ - t } \omega ) \in B( \theta_{ - t } \omega ) $,
	\begin{equation*}
		\Vert \psi (t,\theta_{ - t} \omega,\psi_0(\theta_{ - t} \omega))\Vert^2_{\cH}  \leq c_1 (r _{1} ^{ \varepsilon } ( \omega ) + 1), \quad \forall \ t \geq T_{3} ,
	\end{equation*}
	and
	\begin{equation*}
		\Vert \phi (t,\theta_{ - t} \omega,\phi_0(\theta_{ - t} \omega))\Vert^2_{\cH}  \leq c_2 (r _{2} ^{ \varepsilon } ( \omega )  + 1), \quad \forall \ t \geq T_{3} ,
	\end{equation*}
	where $ c_{1}, \  c_{2} $ are positive deterministic constants independent of $ \varepsilon $, $ r _{2} ^{ \varepsilon } ( \omega ) = r _{1} ^{ \varepsilon } ( \omega ) + r(\omega) $ and $ r(\omega) $ is the tempered function in \eqref{s1}.
\end{lemma}

\begin{proof}
	Substituting $ \psin $ in \eqref{conclusion1} by $ \psi $, we easily get that
	\begin{equation*}
		\Vert \psi (t,\theta_{ - t} \omega,\psi_0(\theta_{ - t} \omega))\Vert^2_{\cH}  \leq c_1 (r _{1} ^{ \varepsilon } ( \omega ) + 1), \quad \forall \ t \geq T_{3}.
	\end{equation*}
	Relation \eqref{phi} between $ \psi $ and $ \phi $ implies that
	\begin{equation*}
		\begin{aligned}
			\Vert \psi_0(\theta_{ - t} \omega) \Vert _{\cH}^2
			& = \Vert \phi_0(\theta_{ - t} \omega) - (\varepsilon y(\omega),0)\Vert _{\cH}^2\\
			& \leq 2\Vert \phi_0(\theta_{ - t} \omega) \Vert 	_{\cH}^2 + 2\varepsilon ^2 \Vert y(\omega) \Vert^2 \\
			& \leq 2\Vert \phi_0(\theta_{ - t} \omega) \Vert _{\cH}^2 + 2c\varepsilon ^2 \vert z(\omega) \vert^2.
		\end{aligned}
	\end{equation*}
	We know that $\phi_0( \theta_{ - t} \omega ) \in B( \theta_{ - t} \omega ) \subset \cD $ is tempered and $ z( \omega ) $ is also tempered, then $ \psi_0( \theta_{ - t} \omega ) $ is tempered. Therefore, by \eqref{s1}, \eqref{sv}, \eqref{phi} and \eqref{conclusion1}, we obtain that, for all $t \geq T_3 $,
	\begin{equation}\label{conclusion2}
		\begin{aligned}
			\Vert \phi (t,\theta_{ - t} \omega,\phi_0(\theta_{ - t} \omega))\Vert ^2_{\cH}
			& = \Vert \psi (t,\theta_{ - t} \omega,\psi_0(\theta_{ - t} \omega))
			+ (\varepsilon y(\omega),0)\Vert ^2_{\cH} \\
			& \leq 2 \Vert \psi (t,\theta_{ - t} \omega,\psi_0(\theta_{ - t} \omega))\Vert ^2_{\cH}
			+ 2\varepsilon ^2 \Vert y(\omega) \Vert^2 \\
			& \leq 2 c_1 (r _{1} ^{ \varepsilon } ( \omega )  + 1) + 2c \varepsilon ^2 \vert z(\omega) \vert^2 \\
			& \leq c_2 (r _{1} ^{ \varepsilon } ( \omega )  + \varepsilon^{2} r(\omega) + 1) \\
			& =: c_2 (r _{2} ^{ \varepsilon } ( \omega ) + 1),
		\end{aligned}
	\end{equation}
	where $ c_2 $ is a constant independent of $ \varepsilon $. This completes the proof.
\end{proof}
\begin{corollary}\label{rem1}
	Notice that $ \Phi( t , \omega ) \phi_0( \omega ) = \phi( t , \omega , \phi_0( \omega ) )  $. By \eqref{conclusion2}, we know that for all $ t \geq T_{3} $,
	\begin{equation}\label{Phi}
		\Vert \Phi( t , \theta_{ - t } \omega ) \phi_0( \theta_{ - t } \omega ) \Vert _{\cH}^2
		= \Vert \phi( t , \theta_{ - t } \omega , \phi_0( \theta_{ - t } \omega ) ) \Vert _{\cH}^2
		\leq c_2 (r _{2} ^{ \varepsilon } ( \omega ) + 1).
	\end{equation}
	Given $ \omega \in \Omega $, we denote
	\begin{equation}\label{K_omega}
		K(\omega) = \left\{ \phi \in \cH(\cO) : \Vert \phi \Vert ^2 _{\cH} \leq c_2 (r _{2} ^{ \varepsilon } ( \omega ) + 1) \right\}.
	\end{equation}
	It is clear that $\{K(\omega)\}_{\omega \in \Omega} \in \cD$. Moreover, \eqref{Phi} indicates that $ \{K(\omega)\}_{\omega \in \Omega} $ is a random absorbing set for $\Phi$ in $\cD$.
\end{corollary}

To prepare for the proof of the compactness of $ \Phi $, we decompose the system into  two subproblem (see \cite{GP2000,KT2009,Temam1988}): one decays exponentially and the other is bounded in a higher regular space. Since $ \psi = \tran{v}{\eta} $, we split the solution as $ \psi = \psiL + \psiN = \tran{\vL}{\etaL} + \tran{\vN}{\etaN} $. Also, $ \phi = \phiL + \phiN $. Then we have
\begin{equation}\label{linear}
	\left\{
	\begin{aligned}
		& \pt_{t} \vL + \nu A \vL + \int_{0}^{\infty} \mu (s) A \etaL \d s + B(v + \varepsilon y, \vL) = 0,\\
		& \pt_{t} \etaL = T \etaL + \vL,\\
		& v_{0L} = v_{0},\ \eta_{0L} = \eta^{0}.
	\end{aligned}
	\right.
\end{equation}
and
\begin{equation}\label{nonlinear}
	\left\{
	\begin{aligned}
		& \begin{aligned}
			& \pt_{t} \vN + \nu A \vN + \int_{0}^{\infty} \mu (s) A \etaN \d s \\
			& \qquad \qquad + B(v + \varepsilon y, \vN + \varepsilon y)
			= f + \varepsilon(\sigma y - \nu A y),
		\end{aligned}\\
		& \pt_{t} \etaN = T \etaN + \vN + \varepsilon y,\\
		& v_{0N} = 0,\ \eta_{0N} = 0.
	\end{aligned}
	\right.
\end{equation}

First, we show that $ \psiL $ and $ \phiL $ has an exponential decay.
\begin{lemma}\label{decompose}
	For every $ B(\omega)_{\omega \in \Omega} \in \cD $ and for $ \bP $-$ a.e. $ $ \omega \in \Omega $, the solutions of problem \eqref{linear} satisfy the following exponential decay property, i.e., for $ t > 0 $ and $ \psi_{0L} (\theta_{ - t } \omega) \in B(\theta_{ - t } \omega) $,
	\begin{equation}\label{exp1}
		\norm{\psiL(t , \theta_{ - t } \omega , \psi_{0L}(\theta_{ - t } \omega))}^{2}_{\cH} \leq e ^{- 2 \delta_{0} t} \norm{\psi_{0L}}^{2}_{\cH}.
	\end{equation}
	What's more,
	\begin{equation}\label{exp2}
		\norm{\phiL(t , \theta_{ - t } \omega , \phi_{0L}(\theta_{ - t } \omega))}^{2}_{\cH} \leq e ^{- 2 \delta_{0} t} \norm{\phi_{0L}}^{2}_{\cH},
	\end{equation}
	where $ \phi_{0L} = \psi_{0L} + \tran{\varepsilon y}{0} $.
\end{lemma}
\begin{proof}
	Taking $ \cH $ inner product with system \eqref{linear} by $ \psiL $ and adding the equations, we have
	\begin{equation*}
		\onehalf \frac{\d}{\d t}\left( \norm{\vL}^{2} + \normM{\etaL}^{2} \right) + \nu \norm{A \vL}^{2} = \onehalf \int_{0}^{\infty} \mu'(s) \norm{A^\onehalf \etaL}^{2} \d s.
	\end{equation*}
	Same procedure as in Section \ref{wellposd} tells us that
	$$
		\frac{\d}{\d t} \norm{\psiL}^{2}_{\cH} + 2 \delta_{0} \norm{\psiL}^{2}_{\cH} \leq 0.
	$$
	Then by the Gronwall's inequality, we get exponential decay \eqref{exp1}. Note that $ \phiL = \psiL $, we obtain \eqref{exp2}.
\end{proof}

\begin{remark}\label{phiL}
	\eqref{exp2} shows that $ \phiL $ goes to 0 as $ t \rightarrow \infty $, i.e., given $ \zeta > 0 $, there is a sufficient large $ T^* $ such that for $ t \geq T^* $,
	$$
		\norm{\phiL(t , \theta_{ - t } \omega , \phi_{0L}(\theta_{ - t } \omega))}^{2}_{\cH} \leq \zeta.
	$$
\end{remark}

Next, we give higher order estimates for $ \psiN $ and $ \phiN $.
\begin{lemma}\label{lem2}
	For every $ B(\omega)_{\omega \in \Omega} \in \cD $ and for  $ \bP $-$ a.e. $ $ \omega \in \Omega $, there exist $ T_{4} = T_{4}(B , \omega) > 0 $ and a tempered function $ r _{3} ^{ \varepsilon } ( \omega ) $ such that for all $ \psi_0 ( \theta_{ - t } \omega ) \in B( \theta_{ - t } \omega ) $,
	\begin{equation*}
		\Vert \psiN (t, \theta_{ - t} \omega, \psi_{0N}(\theta_{ - t}  \omega)) \Vert^2_{\cH^1} \leq c_3 (r _{3} ^{ \varepsilon } ( \omega ) + 1), \quad \forall \ t \geq T_{4} ,
	\end{equation*}
	and
	\begin{equation*}
		\Vert \phiN (t,\theta_{ - t} \omega,\phi_{0N}(\theta_{ - t}  \omega))\Vert^2_{\cH^1} \leq c_4 (r _{4} ^{ \varepsilon } ( \omega ) + 1), \quad \forall \ t \geq T_{4} ,
	\end{equation*}
	where $ c_{3}, \  c_{4} $ are positive deterministic constants independent of $ \varepsilon $, $ r _{4} ^{ \varepsilon } ( \omega ) = r _{3} ^{ \varepsilon } ( \omega ) + r(\omega) $ and $ r(\omega) $ is the tempered function in \eqref{s1}.
\end{lemma}

\begin{proof}
	Taking the inner product with system \eqref{nonlinear} by $ A \psiN = \tran{A \vN}{A \etaN}$, we have
	\begin{equation}\label{EEnergy1}
		\begin{aligned}
			& \frac{1}{2} \frac{\d}{\d t} \left( \norm{\vN}^2_{1} + \norm{\etaN} ^2 _{\mathcal{M}^1} \right) + \nu \norm{A \vN}^2 \\
			& \quad = \frac{1}{2} \int_{0}^{\infty} \mu'(s) \norm{A \etaN} ^2 \d s + \varepsilon \innerMs{\etaN}{y(\theta_t \omega)} \\
			& \quad \quad -  b(v + \varepsilon y(\theta_t \omega), v + \varepsilon y(\theta_t \omega) , A \vN)\\
			& \quad \quad + (f + \varepsilon \left( \sigma y(\theta_t \omega) - \nu A y(\theta_t \omega)\right) , A \vN)\\
			& \quad =: J_1 + J_2 + J_3 + J_4.
		\end{aligned}
	\end{equation}
	It is clear from the Poincar\'{e} inequality that
	\begin{equation}\label{ee0}
		\nu \norm{A \vN}^2 \geq \frac{\nu}{2} \norm{A \vN}^2
		+ \frac{\nu \lambda_1}{2} \norm{\vN}	_{1}^2.
	\end{equation}
	
	Now we are ready to estimate $ J_i, i = 1,2,3,4 $ respectively.
	For $J_1$, from the Dafermos condition \eqref{mu}, we get
	\begin{equation}\label{ee1}
		J_1 \leq -\frac{\delta}{2} \int_{0}^{\infty} \mu (s) \norm{A \etaN}^2 \d s = -\frac{\delta}{2} \norm{\etaN} ^2 _{\mathcal{M}^{1}}.
	\end{equation}
	By utilizing the H\"{o}lder's inequality, the Young's inequality with $ \epsilon $ and \eqref{sv}, one obtains that
	\begin{equation}\label{ee2}
		\begin{aligned}
			J_2
			& = \varepsilon \int_{0}^{\infty} \mu (s) (A \etaN,A y(\theta_t\omega))\d s\\
			& \leq \frac{\delta}{4}\int_{0}^{\infty} \mu (s) \norm{A \etaN} ^2\d s
			+ \frac{\varepsilon^2}{\delta} \int_{0}^{\infty} \mu (s) \Vert A y(\theta_t\omega) \Vert ^2\d s\\
			& \leq \frac{\delta}{4} \norm{\etaN} ^2 _{\mathcal{M}^{1}}
			+ \frac{c \varepsilon^2 \kappa}{\delta}  \beta_{1}(\theta_{ t } \omega).
		\end{aligned}
	\end{equation}
	A direct calculation deduces that
	\begin{equation*}
		\begin{aligned}
			\vert J_3 \vert
			& \leq \abs{b(\vL , \vN , A \vN)}
			+ \abs{b(\vL , \varepsilon y(\theta_t \omega) , A \vN)}\\
			& \quad + \abs{b(\vN , \vN , A \vN)}
			+ \abs{b(\vN , \varepsilon y(\theta_t \omega) , A \vN)}\\
			& \quad + \abs{b(\varepsilon y(\theta_t \omega) , \vN , A \vN)}
			+ \abs{b(\varepsilon y(\theta_t \omega) , \varepsilon y(\theta_t \omega) , A \vN)}.
		\end{aligned}
	\end{equation*}
	It follows from the H\"{o}lder's inequality, the Young's inequality with $ \epsilon $, Lemma \ref{b} and \eqref{sv} that
	\begin{align*}
		\vert b(\vL , \vN , A \vN) \vert
		& \leq \hat{c} \norm{\vL}^{\onehalf} \norm{A^{\onehalf} \vL}^{\onehalf} \norm{A^{\onehalf} \vN}^{\onehalf} \norm{A \vN}^{\frac{3}{2}}  \\
		& \leq \frac{ \nu }{48} \norm{A \vN} ^{2} +
		\left( \frac{11664 \hat{c}^4}{ \nu ^3 } \norm{\vL}^2 \norm{\vL}_1^2 \right) \norm{\vN}_1^2 ,
	\end{align*}
	\begin{align*}
		\vert b(\vL , \varepsilon y(\theta_t \omega) , A \vN) \vert
		& \leq \varepsilon \hat{c} \norm{\vL}^{\onehalf} \norm{A^{\onehalf} \vL}^{\onehalf} \norm{A^{\onehalf} \varepsilon y(\theta_t \omega)}^{\onehalf} \norm{A \varepsilon y(\theta_t \omega)}^{\onehalf} \norm{A \vN}  \\
		& \leq \frac{ \nu }{48} \norm{A \vN} ^{2} + 6 \varepsilon^2 c_0 \norm{\vL} \norm{\vL}_1 \beta_{1}(\theta_{ t } \omega) ,
	\end{align*}
	\begin{align*}
		\vert b(\vN , \vN , A \vN) \vert
		& \leq \hat{c} \norm{\vN}^{\onehalf} \norm{A^\onehalf \vN} \norm{A \vN} ^{ \frac{3}{2} }  \\
		& \leq \frac{ \nu }{48} \norm{A \vN} ^{2} + \left( \frac{11664\hat{c}^4}{ \nu ^3 } \norm{\vN}^2 \norm{\vN}_1^2 \right) \norm{\vN}_1^2 ,
	\end{align*}
	\begin{align*}
		\vert b(\vN , \varepsilon y(\theta_t \omega) , A \vN) \vert
		& \leq \varepsilon \hat{c} \norm{\vN}^\onehalf \norm{A^\onehalf \vN}^\onehalf \norm{A^\onehalf y ( \theta_t \omega )}^\onehalf \norm{A y ( \theta_t \omega )}^\onehalf \norm{A \vN}  \\
		& \leq \frac{ \nu }{48} \norm{A \vN} ^{2} +
		+ \frac{\nu \lambda_{1}}{4} \norm{\vN}_1^2 + \frac{ 144 (\varepsilon \hat{c} \tilde{c})^4 }{\nu^3 \lambda_{1}} \beta_{1}(\theta_{ t } \omega) \norm{\vN} ^2,
	\end{align*}
	\begin{align*}
		\vert b(\varepsilon y(\theta_t \omega) , \vN , A \vN) \vert
		& \leq \varepsilon \hat{c} \norm{y ( \theta_t \omega )}^\onehalf \norm{A^\onehalf y ( \theta_t \omega )}^\onehalf \norm{A^\onehalf \vN}^\onehalf \norm{A \vN}^{\frac{3}{2}}  \\
		& \leq \frac{ \nu }{48} \norm{A \vN} ^{2} +  \frac{ 11664 (\varepsilon \hat{c} \tilde{c})^4 }{\nu ^3} \beta_{1}(\theta_{ t } \omega) \norm{\vN}_1^2
	\end{align*}
	and
	\begin{align*}
		\vert b(\varepsilon y(\theta_t \omega) , \varepsilon y(\theta_t \omega) , A \vN) \vert
		& \leq \varepsilon ^2 \hat{c} \norm{ y ( \theta_t \omega ) }^\onehalf \norm{A^\onehalf y ( \theta_t \omega ) } \norm{A y ( \theta_t \omega ) }^\onehalf \norm{A \vN}  \\
		& \leq \frac{ \nu }{48} \norm{A \vN} ^{2} + \frac{ 8 (\varepsilon \hat{c})^4 \tilde{c}^2 }{\nu} \beta_{1}(\theta_{ t } \omega).
	\end{align*}
	Hence
	\begin{equation}\label{ee3}
		\begin{aligned}
			\vert J_3 \vert
			& \leq \frac{ \nu }{8} \norm{A \vN} ^{2}
			+ \frac{\nu \lambda_{1}}{4} \norm{\vN}_1^2 + c \varepsilon^{2} \beta_{1}(\theta_{ t } \omega) \left( \norm{\vN} ^2 + \norm{\vL} \norm{\vL}_1 \right) \\
			& \quad + \left( c_5 \beta_{1}(\theta_{ t } \omega) + c_6 \norm{\vN}^2 \norm{\vN}_1^2 + c_6 \norm{\vL}^2 \norm{\vL}_1^2 \right) \norm{\vN}_1^2.
		\end{aligned}
	\end{equation}
	where $ c_5 = \frac{ 11664 ( \hat{c} \tilde{c})^4 }{\nu ^3} $ and $ c_6 = \frac{11664 \hat{c}^4}{ \nu ^3 } $.\\
	
	Since $ f \in \bm{H} $, it can be deduced from the Young's inequality with $ \epsilon $ that
	\begin{equation}\label{ee4}
		\begin{aligned}
			J_4
			& \leq \frac{ \nu }{8} \norm{A \vN} ^2  + \frac{2}{\nu} \left(\Vert f \Vert ^2 + 2c \varepsilon^{2} \left(\Vert y(\theta_t \omega) \Vert ^2 + \Vert A y(\theta_t \omega) \Vert ^2\right) \right) \\
			& \leq \frac{ \nu }{8} \norm{A \vN} ^2 + c \left( 1 + \varepsilon^{2} \beta_{1}(\theta_{ t } \omega) \right).
			\end{aligned}
	\end{equation}
	Then by adding \eqref{EEnergy1}, \eqref{ee1}--\eqref{ee4}, we have
	\begin{equation*}
		\begin{aligned}
			& \quad \frac{\d}{\d t} \Vert \psiN (t,\omega,\psi_{0N}(\omega))\Vert^2_{\cH ^1} + \frac{\nu}{2} \norm{A \vN}^2 \\
			& \leq \left( - \delta_0 + c_5 \varepsilon^4 \beta_{1}(\theta_{ t } \omega) + c_6 \norm{\vN}^2 \norm{\vN}_1^2 \right) \Vert \psiN (t,\omega,\psi_{0N}(\omega))\Vert^2_{\cH ^1} \\
			& \quad + c \left(1 + \varepsilon^{2} \beta _1 (\theta_t \omega) + \varepsilon^{2} \beta_{1}(\theta_{ t } \omega) (\norm{\vN} ^2 + \norm{\vL} \norm{\vL}_1) \right).
		\end{aligned}
	\end{equation*}
	Since $ (\vL + \vN) \in L^{\infty}(T_{3},\infty; \bH) \cap L^{2}(T_{3},\infty; \bV) $ from \eqref{unifb1} and $ \psi_{0N} = \tran{0}{0} \in \cH^{1} $, an argument similar to the one used in Lemma \ref{lem1} shows that there exists a $ T_{4} = T_{4}( B, \omega ) \geq T_{3} $ such that for all $ t \geq T_{4} $,
	\begin{equation}\label{conclusion3}
		\Vert \psiN (t,\theta_{ - t} \omega,\psi_{0N}(\theta_{ - t}  \omega))\Vert^2_{\cH^1} \leq c_3 (r _{3} ^{ \varepsilon } ( \omega ) + 1),
	\end{equation}
	and
	\begin{equation}\label{conclusion4}
		\Vert \phiN (t,\theta_{ - t} \omega,\phi_{0N}(\theta_{ - t}  \omega))\Vert^2_{\cH^1} \leq c_4 (r _{4} ^{ \varepsilon } ( \omega ) + 1),
	\end{equation}
	where
	$$
		r_{3} ^{ \varepsilon } ( \omega ) = \int_{-\infty}^{0} e ^{ \delta_0 s + c_5 \varepsilon^{4} \int_{s}^{0} \beta _1 (\theta_{\tau} \omega) \d \tau } \left( 1 + \varepsilon^{2} \beta _1 (\theta_s \omega) + c_1 \varepsilon^{2} \beta _1 (\theta_s \omega) (r_{1}^\varepsilon (\omega) + 1) \right)  \d s.
	$$
	It is easy to show that $ r_{3} ^{ \varepsilon } ( \omega ) $ is a tempered random variable (Analogously to $ r _{1} ^{ \varepsilon } (\omega) $) and $ r _{4} ^{ \varepsilon } ( \omega ) = r _{3} ^{ \varepsilon } ( \omega ) + \varepsilon ^{2} r( \omega ) $.
	This completes the proof.
\end{proof}

We denote $ Q_{R} = \{ x \in \cO : \vert x \vert < R \} $ and $ Q_{R}^c = \cO \setminus Q_{R} $. Note that the embedding $ \cH^{1} \hookrightarrow \cH $ is not compact any more in unbounded domains, it is hard to prove the exsitence of uniqueness of the random attractor. Inspired by \cite{BLW2009,Wang1999,Wang2009,Wang2014}, we obtain the far-field values of solutions which can be applied to get the asymptotic compactness in unbounded domains.
\begin{lemma}\label{lem3}
	Suppose that $ B = \left\{ B( \omega ) \right\}_{\omega \in \Omega} \in \cD $ and $ \psi_0( \omega ) \in B( \omega ) $. Then for every $ \zeta > 0 $ and $ \bP $-$ a.e. $ $ \omega \in \Omega $, there exist a $ T_{7} = T_{7} ( B , \omega , \zeta ) $ and a $ R_{3} = R_{3} ( \omega , \zeta ) $ such that for all $ t \geq T_{7} $,
	\begin{equation}\label{psi_zeta}
		\Vert \psi ( t, \theta_{ - t} \omega, \psi_0( \theta_{ - t} \omega ) ) \Vert^2 _{ \cH ( Q_{R_{3}}^{c} ) } \leq \zeta.
	\end{equation}
\end{lemma}

Bates, Lu and Wang \cite{BLW2009} introduced a cutoff function in $ \cO $ for the first step. Let $ \rho $ be a smooth function defined on $ \bR ^{+} $ such that $ 0 \leq \rho (s) \leq 1 $ for all $ s \in \bR ^{+} $, and
$$
	\rho (s) =
	\left\{
	\begin{aligned}
		& 0 && 0 < s \leq 1, \\
		& 1 && s \geq 2.
	\end{aligned}
	\right.
$$
Then $ \rho ' ( s ) = 0 $ for all $ s \in \bR^{+}\backslash [ 1 , 2 ] $.

It is difficult for us to keep going like \cite{BLW2009} since we can not contruct the cutoff energy without $ \beta u $ immediately, that is, the Poincar\'{e} inequality does not work for $ - \Delta u $ when the energy is cutoff (i.e., weighted). In our work, we will show a generalized Poincar\'{e} inequality. To prove it, we suppose that
$$ \vert \rho ' (s) \vert \leq \sqrt{\lambda_1} \rho(s) \leq \sqrt{\lambda_1} $$
for all $ s \in \bR ^{+} $. The assumption here is different from that in \cite{BLW2009} but reasonable because $ \rho $ is a smooth function.
For our convenience, we denote
$$
	\Vert u \Vert _{L^{2}_{\rho}} ^{2} = \int_{ \cO } \vert u \vert ^2 \rho \left( \frac{ \vert x \vert ^2 }{ k^2 } \right) \d x, \quad \forall \ u \in \bm{H},
$$
where $ k $ is a large positive constant.
Thus we have the following generalized Poincar\'{e} inequality.

\begin{lemma}\label{poincare}
	For all $ u \in \bm{V} $, we have
	\begin{equation*}
		\frac{\lambda_1}{4} \Vert u \Vert _{L^{2}_{\rho}} ^{2} \leq \Vert \nabla u \Vert _{L^{2}_{\rho}} ^{2},
	\end{equation*}
	where $ \lambda_1 $ is the constant in the Poincar\'{e} inequality.
\end{lemma}
\begin{proof}
	First, we show that $ \rho ^{\frac{1}{2}} \left( \frac{ \vert x \vert ^2 }{ k^2 } \right) u \in \bm{V} $, that is $ \left\Vert \rho ^{\frac{1}{2}} \left( \frac{ \vert x \vert ^2 }{ k^2 } \right) u \right\Vert_{\bm{V}} < \infty $. We know that
	$$
		\abs{\nabla \rho ^{\frac{1}{2}} ( s )}^{2} = \frac{1}{4} \abs{\rho ^{- \frac{1}{2}} ( s ) \rho ' ( s )}^{2} \leq \frac{\lambda_1}{4} \rho (s) \leq \frac{\lambda_1}{4}.
		$$
	Since $ u \in \bm{V} $, it follows that
	\begin{align*}
		& \quad \int_{\cO} \rho \left( \frac{ \vert x \vert ^2 }{ k^2 } \right) \vert u \vert ^2 \d x
		+ \int_{\cO} \left\vert \nabla \left( \rho ^{\frac{1}{2}} \left( \frac{ \vert x \vert ^2 }{ k^2 } \right) u \right) \right\vert ^2 \d x \\
		& \leq \int_{\cO} \rho \left( \frac{ \vert x \vert ^2 }{ k^2 } \right) \vert u \vert ^2 \d x
		+ \int_{\cO} \left\vert \nabla \left( \rho ^{\frac{1}{2}} \left( \frac{ \vert x \vert ^2 }{ k^2 } \right) \right) u \right\vert ^2 \d x
		+ \int_{\cO}  \rho  \left( \frac{ \vert x \vert ^2 }{ k^2 } \right) \left\vert \nabla u \right\vert ^2 \d x \\
		& \leq \int_{\cO} \vert u \vert ^2 \d x
		+ \frac{\lambda_1}{4} \int_{\cO} \left\vert u \right\vert ^2 \d x
		+ \int_{\cO} \left\vert \nabla u \right\vert ^2 \d x < \infty .
	\end{align*}
	Then $ \rho ^{\frac{1}{2}} \left( \frac{ \vert x \vert ^2 }{ k^2 } \right) u \in \bm{V} $. By applying the Poincar\'{e} inequality, we get
	\begin{align*}
		\lambda_1 \int_{\cO} \rho \left( \frac{ \vert x \vert ^2 }{ k^2 } \right) \vert u \vert ^2 \d x
		& \leq \int_{\cO} \left\vert \nabla \left( \rho ^{\frac{1}{2}} \left( \frac{ \vert x \vert ^2 }{ k^2 } \right) u \right) \right\vert ^2 \d x \\
		& \leq \frac{\lambda_1}{2} \int_{\cO} \rho \left( \frac{ \vert x \vert ^2 }{ k^2 } \right) \left\vert u \right\vert ^2 \d x
		+ 2 \int_{\cO}  \rho  \left( \frac{ \vert x \vert ^2 }{ k^2 } \right) \left\vert \nabla u \right\vert ^2 \d x.
	\end{align*}
	Hence
	\begin{equation*}
		\frac{\lambda_1}{4} \int_{\cO} \rho \left( \frac{ \vert x \vert ^2 }{ k^2 } \right) \vert u \vert ^2 \d x
		\leq \int_{\cO}  \rho  \left( \frac{ \vert x \vert ^2 }{ k^2 } \right) \left\vert \nabla u \right\vert ^2 \d x,
	\end{equation*}
	which completes the proof.
\end{proof}

\begin{proof}[Proof of Lemma \ref{lem3}]
	Taking $ w = \rho \left( \frac{ \vert x \vert ^2 }{ k^2 } \right) v $ in \eqref{SNSM5}, we have
	\begin{equation}\label{ce1}
		\begin{aligned}
			& \frac{1}{2} \frac{\d}{\d t} \int_{\cO} \rho \left( \frac{ \vert x \vert ^2 }{ k^2 } \right) \vert v \vert ^2 \d x
			+ \nu \int_{\cO} A v \cdot \rho \left( \frac{ \vert x \vert ^2 }{ k^2 } \right) v \d x
			+ \int_{0}^{\infty} \mu (s) \int_{\cO} A \eta \cdot \rho \left( \frac{ \vert x \vert ^2 }{ k^2 } \right) v \d x \d s \\
			& \quad + b ( v + \varepsilon y, v + \varepsilon y, \rho \left( \frac{ \vert x \vert ^2 }{ k^2 } \right) v  )
			= \int_{\cO} \left( f(x) + \varepsilon \sigma y ( \theta_t \omega ) - \varepsilon A y ( \theta_t \omega ) \right) \cdot \rho \left( \frac{ \vert x \vert ^2 }{ k^2 } \right) v  \d x.
		\end{aligned}
	\end{equation}
	It can be deduced from the Young's inequality and Lemma \ref{poincare} that
	\begin{align*}
		K_1
		:& = \nu \int_{\cO} A v \cdot \rho \left( \frac{ \vert x \vert ^2 }{ k^2 } \right) v \d x \\
		& = \nu \int_{\cO} \vert \nabla v \vert ^2 \rho \left( \frac{ \vert x \vert ^2 }{ k^2 } \right) \d x
		+ \nu \int_{\cO} v \rho ' \left( \frac{ \vert x \vert ^2 }{ k^2 } \right) \frac{ 2 x }{ k^2 } \cdot \nabla v \d x  \\
		& \geq \frac{\nu}{2} \Vert \nabla v \Vert _{L^{2}_{\rho}}^{2}
		+ \frac{\nu \lambda_1}{8} \Vert v \Vert _{L^{2}_{\rho}}^{2} + \nu \int_{ k \leq \vert x \vert \leq \sqrt{2} k } v \rho ' \left( \frac{ \vert x \vert ^2 }{ k^2 } \right) \frac{ 2 x }{ k^2 } \cdot \nabla v \d x.
	\end{align*}
	For the last term of the above inequality, we get from the Young's inequality that
	\begin{align*}
		& \quad \nu \left\vert \int_{ k \leq \vert x \vert \leq \sqrt{2} k } v 	\rho ' \left( \frac{ \vert x \vert ^2 }{ k^2 } \right) \frac{ 2 x }{ k^2 } \cdot \nabla v \d x \right\vert \\
		& \leq \frac{2\sqrt{2}\nu}{k} \int_{ k \leq \vert x \vert \leq \sqrt{2} k } |v| \left\vert \rho ' \left( \frac{ \vert x \vert ^2 }{ k^2 } \right) \right\vert \left\vert \nabla v \right\vert \d x \\
		& \leq \frac{ 2 \sqrt{2\lambda_1} \nu }{k} \int_{ \cO } \rho \left( \frac{ \vert x \vert ^2 }{ k^2 } \right) |v| \left\vert \nabla v \right\vert \d x \\
		& \leq \frac{\sqrt{2\lambda_1} \nu }{k} \left( \Vert v \Vert ^2 + \Vert \nabla v \Vert ^2_{L^{2}_{\rho}} \right).
	\end{align*}
	Therefore, we find that
	\begin{equation}\label{K1}
		K_1
		\geq \frac{\nu}{4} \Vert \nabla v \Vert _{L^{2}_{\rho}}^{2}
		+ \frac{\nu \lambda_1}{8} \Vert v \Vert _{L^{2}_{\rho}}^{2}
		- \frac{\sqrt{2\lambda_1} \nu}{k} \Vert v \Vert ^2,
	\end{equation}
	where we have chosen $ k \geq R_{0} := 4\sqrt{2 \lambda_{1}} $.
	
	Next, for the third term of the left-hand side of \eqref{ce1}, it follows from the Young's inequality with $ \epsilon $ that
	\begin{align}\label{K2}
		& K_2 := \int_{0}^{\infty} \mu (s) \int_{\cO} A \eta \cdot \rho \left( \frac{ \vert x \vert ^2 }{ k^2 } \right) v \d x \d s \nonumber \\
		& = \int_{0}^{\infty} \mu (s) \int_{\cO} A \eta \cdot \rho \left( \frac{ \vert x \vert ^2 }{ k^2 } \right) \left( \pt_t \eta + \pt_s \eta - \varepsilon y( \theta_t \omega ) \right) \d x \d s \nonumber \\
		& = \frac{1}{2} \frac{\d}{\d t} \int_{0}^{\infty} \mu (s) \int_{ \cO } \rho \left( \frac{ \vert x \vert ^2 }{ k^2 } \right) \vert \nabla \eta \vert ^2 \d x \d s
		+ \int_{0}^{\infty} \mu (s) \int_{\cO} \nabla \eta \rho \left( \frac{ \vert x \vert ^2 }{ k^2 } \right) \cdot \pt_s \nabla \eta \d x \d s \nonumber \\
		& \quad - \varepsilon \int_{0}^{\infty} \mu (s) \int_{\cO} \nabla \eta \rho \left( \frac{ \vert x \vert ^2 }{ k^2 } \right) \cdot \nabla y( \theta_t \omega ) \d x \d s \nonumber \\
		& \qquad + \int_{0}^{\infty} \mu (s) \int_{\cO}  v \rho ' \left( \frac{ \vert x \vert ^2 }{ k^2 } \right) \frac{2x}{k^2} \cdot \nabla \eta \d x \d s \nonumber \\
		& \geq \frac{1}{2} \frac{\d}{\d t} \int_{0}^{\infty} \mu (s) \int_{ \cO } \rho \left( \frac{ \vert x \vert ^2 }{ k^2 } \right) \vert \nabla \eta \vert ^2 \d x \d s
		- \frac{1}{2} \int_{0}^{\infty} \mu ' (s) \int_{ \cO } \rho \left( \frac{ \vert x \vert ^2 }{ k^2 } \right) \vert \nabla \eta \vert ^2 \d x \d s \nonumber\\
		& \quad - \varepsilon \int_{0}^{\infty} \mu (s) \int_{\cO} \rho \left( \frac{ \vert x \vert ^2 }{ k^2 } \right) \vert \nabla \eta \vert \vert \nabla y( \theta_t \omega ) \vert \d x \d s\nonumber \\
		& \qquad
		- \frac{ \sqrt{2\lambda_1} }{k} \int_{0}^{\infty} \mu (s) \int_{\cO} \left( \vert \nabla \eta \vert ^2 + \vert v \vert ^2 \right) \d x \d s \nonumber \\
		& \geq \frac{1}{2} \frac{\d}{\d t} \int_{0}^{\infty} \mu (s) \int_{ \cO } \rho \left( \frac{ \vert x \vert ^2 }{ k^2 } \right) \vert \nabla \eta \vert ^2 \d x \d s
		+ \frac{\delta}{2} \int_{0}^{\infty} \mu (s) \int_{ \cO } \rho \left( \frac{ \vert x \vert ^2 }{ k^2 } \right) \vert \nabla \eta \vert ^2 \d x \d s \nonumber \\
		& \quad - \frac{\delta}{4} \int_{0}^{\infty} \mu (s) \int_{\cO} \rho \left( \frac{ \vert x \vert ^2 }{ k^2 } \right) \vert \nabla \eta \vert ^2 \d x \d s\nonumber \\
		& \qquad
		- \frac{\varepsilon^2}{\delta} \int_{0}^{\infty} \mu (s) \int_{\cO} \rho \left( \frac{ \vert x \vert ^2 }{ k^2 } \right) \vert \nabla y( \theta_t \omega ) \vert ^2 \d x \d s
		- \frac{ \sqrt{2\lambda_1} }{k} \norm{\eta}_{ \mathcal{M} } ^2
		- \frac{ \sqrt{2\lambda_1} \kappa }{k} \Vert v \Vert ^2 \nonumber \\
		& = \frac{1}{2} \frac{\d}{\d t} \int_{0}^{\infty} \mu (s) \int_{ \cO } \rho \left( \frac{ \vert x \vert ^2 }{ k^2 } \right) \vert \nabla \eta \vert ^2 \d x \d s
		+ \frac{\delta}{4} \int_{0}^{\infty} \mu (s) \int_{ \cO } \rho \left( \frac{ \vert x \vert ^2 }{ k^2 } \right) \vert \nabla \eta \vert ^2 \d x \d s \nonumber \\
		& \quad - \frac{\varepsilon^2 \kappa}{\delta} \Vert \nabla y( \theta_t \omega ) \Vert _{L^{2}_{\rho}}^{2}
		- \frac{ \sqrt{2\lambda_1} }{k} \norm{\eta}_{ \mathcal{M} } ^2
		- \frac{ \sqrt{2\lambda_1} \kappa }{k} \Vert v \Vert ^2.
	\end{align}
	From the definition and properties of the trilinear form $ b ( \cdot , \cdot , \cdot ) $, we obtain
	\begin{align*}
		\left \vert b( v , v , \rho \left( \frac{ \vert x \vert ^2 }{ k^2 } \right) v ) \right \vert
		& = \left \vert \int_{\cO} ( v \cdot \nabla ) v \cdot \rho \left( \frac{ \vert x \vert ^2 }{ k^2 } \right) v \d x \right \vert \\
		& \leq \left \vert \int_{ \cO } ( v \cdot \nabla ) v \cdot v \d x \right \vert \\
		& = \left \vert b ( v , v , v ) \right \vert= 0
	\end{align*}
	and
	\begin{align*}
		\left \vert b( \varepsilon y , v , \rho \left( \frac{ \vert x \vert ^2 }{ k^2 } \right) v ) \right \vert = 0,
	\end{align*}
	
	For the last term of the left-hand side of \eqref{ce1}, we know that
	\begin{align*}
		\vert K_3 \vert
		: & = \left \vert b ( v + \varepsilon y, v + \varepsilon y, \rho \left( \frac{ \vert x \vert ^2 }{ k^2 } \right) v  ) \right \vert \\
		& \leq \left \vert b( v , \varepsilon y , \rho \left( \frac{ \vert x \vert ^2 }{ k^2 } \right) v ) \right \vert + \left \vert b( \varepsilon y , \varepsilon y , \rho \left( \frac{ \vert x \vert ^2 }{ k^2 } \right) v ) \right \vert \\
		& = : L_{1} + L_{2}.
	\end{align*}
	It follows from the the Weighted H\"{o}lder's inequality and Lemma \ref{b} that
	\begin{align*}
		L_{1}
		= \left \vert \int_{\cO} ( v \cdot \nabla ) ( \varepsilon y ) \cdot \rho \left( \frac{ \vert x \vert ^2 }{ k^2 } \right) v \d x \right \vert
		& \leq c \varepsilon \Vert A^{\frac{1}{2}} v \Vert _{L^{2}_{\rho}} \Vert A^{\frac{1}{2}} y(\theta_t \omega) \Vert _{L^{2}_{\rho}} \Vert v \Vert _{L^{2}_{\rho}} \\
		& \leq c \varepsilon \Vert A^{\frac{1}{2}} v \Vert _{L^{2}_{\rho}} \Vert A^{\frac{1}{2}} y(\theta_t \omega) \Vert _{L^{2}_{\rho}} \Vert v \Vert
	\end{align*}
	and
	\begin{align*}
		L_{2}
		& = \left \vert \int_{\cO} ( \varepsilon y \cdot \nabla ) ( \varepsilon y ) \cdot \rho \left( \frac{ \vert x \vert ^2 }{ k^2 } \right) v \d x \right \vert
		\leq c \varepsilon^{2} \Vert A^{\frac{1}{2}} y(\theta_t \omega) \Vert _{L^{2}_{\rho}} \Vert y(\theta_t \omega) \Vert _{L^{2}_{\rho}} \Vert A^{\frac{1}{2}} v \Vert_{L^{2}_{\rho}}.
	\end{align*}
	Then we see that
	\begin{align}\label{K3}
		\vert K_3 \vert
		\leq \frac{ \varepsilon^{2} }{2} \Vert A^{\frac{1}{2}} y(\theta_t \omega) \Vert _{L^{2}_{\rho}} ^2 \Vert v \Vert^{2}
		+ \frac{\nu}{8} \Vert \nabla v \Vert ^2_{L^{2}_{\rho}} + c \varepsilon^{4} \Vert A^{\frac{1}{2}} y(\theta_t \omega) \Vert _{L^{2}_{\rho}}^{2} \Vert y(\theta_t \omega) \Vert _{L^{2}_{\rho}}^{2}.
	\end{align}
	
	Finally, we derive from the Young's inequality with $ \epsilon $ that
	\begin{align}\label{K4}
		K_4
		: & = \int_{\cO} \left( f(x) + \varepsilon \sigma y ( \theta_t \omega ) - \varepsilon A y ( \theta_t \omega ) \right) \cdot \rho \left( \frac{ \vert x \vert ^2 }{ k^2 } \right) v  \d x \nonumber \\
		& \leq \int_{\cO} \rho \left( \frac{ \vert x \vert ^2 }{ k^2 } \right) \left| f(x) + \varepsilon \sigma y ( \theta_t \omega ) - \varepsilon A y ( \theta_t \omega ) \right| \vert v \vert \d x \\
		& \leq \frac{\nu \lambda_1}{16} \Vert v \Vert _{L^{2}_{\rho}}^{2} + \frac{4}{\nu \lambda_1} \left( \Vert f \Vert _{L^{2}_{\rho}}^{2} + \varepsilon \sigma \Vert y ( \theta_t \omega ) \Vert _{L^{2}_{\rho}}^{2} + \varepsilon \Vert A y ( \theta_t \omega ) \Vert _{L^{2}_{\rho}}^{2} \right). \nonumber
	\end{align}
	Thus, a combination of \eqref{ce1}--\eqref{K4} implies that
	\begin{align} \label{dH}
		& \quad \frac{\d}{\d t} H(t , \omega) + \delta_2 H(t , \omega) + \frac{\nu}{4} \Vert \nabla v \Vert _{L^{2}_{\rho}}^{2}  \nonumber\\
		& \leq \frac{2 \sqrt{2\lambda_1}}{k} \left( \nu + \kappa \right) \Vert v \Vert ^2
		+ \varepsilon^{2} \Vert A^{\frac{1}{2}} y(\theta_t \omega) \Vert _{L^{2}_{\rho}} ^2 \Vert v \Vert^{2} \nonumber \\
		& \quad + \frac{ 2 \sqrt{2\lambda_1} }{k} \norm{\eta}_{ \mathcal{M} } ^2 + 2 c \varepsilon^{4} \Vert A^{\frac{1}{2}} y(\theta_t \omega) \Vert _{L^{2}_{\rho}}^{2} \Vert y(\theta_t \omega) \Vert _{L^{2}_{\rho}}^{2} \nonumber\\
		& \quad + \frac{8}{\nu \lambda_1} \left( \Vert f \Vert _{L^{2}_{\rho}}^{2} + \varepsilon \sigma \Vert y ( \theta_t \omega ) \Vert _{L^{2}_{\rho}}^{2} + \frac{\varepsilon ^2 \kappa \nu \lambda_1}{2 \delta} \Vert A ^{\frac{1}{2}} y ( \theta_t \omega ) \Vert _{L^{2}_{\rho}}^2 + \varepsilon \Vert A y ( \theta_t \omega ) \Vert _{L^{2}_{\rho}}^{2} \right),
	\end{align}
	where $ \delta_2 = \min \left\{ \frac{\nu \lambda_1}{8} , \frac{\delta}{2} \right\} $ and
	\begin{equation}\label{def_H}
		H(t , \omega) = \int_{\cO} \rho \left( \frac{ \vert x \vert ^2 }{ k^2 } \right) \vert v \vert ^2 \d x
		+ \int_{0}^{\infty} \mu (s) \int_{ \cO } \rho \left( \frac{ \vert x \vert ^2 }{ k^2 } \right) \vert \nabla \eta \vert ^2 \d x \d s.
	\end{equation}	
	Rearranging inequality \eqref{dH}, one obtains that
	\begin{equation}\label{ce2}
		\begin{aligned}
			& \quad \frac{\d}{\d t} H(t , \omega) + \delta_2 H(t , \omega) + \frac{\nu}{4} \Vert \nabla v \Vert _{L^{2}_{\rho}}^{2} \leq \frac{c}{k} \Vert \psi \Vert _{\cH}^2
			+ F(t , \theta_t \omega) \Vert \psi \Vert _{\cH}^2
			+ c G(t , \theta_t \omega),
		\end{aligned}
	\end{equation}
	where
	$$
		F(t , \theta_t \omega) = \varepsilon^{2} \Vert A^{\frac{1}{2}} y(\theta_t \omega) \Vert _{L^{2}_{\rho}} ^2,
	$$
	and
		\begin{align*}
			G(t , \theta_t \omega) & = \Vert f \Vert _{L^{2}_{\rho}}^{2} + \Vert y ( \theta_t \omega ) \Vert _{L^{2}_{\rho}}^{2} \\
			& \quad + \Vert A ^{\frac{1}{2}} y ( \theta_t \omega ) \Vert _{L^{2}_{\rho}}^2 + \Vert A y ( \theta_t \omega ) \Vert _{L^{2}_{\rho}}^{2}
			+ \Vert A ^{\frac{1}{2}} y ( \theta_t \omega ) \Vert _{L^{2}_{\rho}}^4.
		\end{align*}
	Multiplying both sides of \eqref{ce2} by $ e ^{ \delta_2 t } $ and integrating it over $ ( T_{5} , t ) $ where $ T_{5} = \max \{ T_3 , T_{4} \} $,  we get that, for all $ t \geq T_{5} $,
	\begin{align*}
		H( t , \omega )
		& \leq e ^{\delta_2 (T_{5} - t)} H( T_{5} , \omega) \\
		& \quad + \frac{c}{k} \int_{T_{5}}^{t} e ^{ \delta_2 ( s - t ) } \Vert \psi ( s , \omega , \psi_0 ( \omega ) ) \Vert _{\cH}^2 \d s \\
		& \quad + \int_{T_{5}}^{t} e ^{ \delta_2 ( s - t ) } F(s , \theta_s \omega) \Vert \psi ( s , \omega , \psi_0 ( \omega ) ) \Vert _{\cH}^2 \d s \\
		& \quad + c \int_{T_{5}}^{t} e ^{ \delta_2 ( s - t ) }  G(s , \theta_s \omega) \d s .
	\end{align*}
	Replacing $ \omega $ by $ \theta_{ - t } \omega $, we see that, for all $ t \geq T_{5} $,
	\begin{align}\label{H}
		H( t , \theta_{ - t } \omega )
		& \leq e ^{\delta_2 (T_{5} - t)} H( T_{5} , \theta_{ - t } \omega) \nonumber \\
		& \quad + \frac{c}{k} \int_{T_{5}}^{t} e ^{ \delta_2 ( s - t ) } \Vert \psi ( s , \theta_{ - t } \omega , \psi_0 ( \theta_{ - t } \omega ) ) \Vert _{\cH}^2 \d s \nonumber \\
		& \quad + \int_{T_{5}}^{t} e ^{ \delta_2 ( s - t ) } F(s , \theta_{ s - t } \omega) \Vert \psi ( s , \theta_{ - t } \omega , \psi_0 ( \theta_{ - t } \omega ) ) \Vert _{\cH}^2 \d s \nonumber \\
		& \quad + c \int_{T_{5}}^{t} e ^{ \delta_2 ( s - t ) }  G(s , \theta_{ s - t } \omega) \d s \nonumber \\
		& = : Y_1 + Y_2 + Y_3 + Y_4.
	\end{align}
	
	Next, we will estimate each terms in \eqref{H}. First, we have from  \eqref{conclusion1} and \eqref{def_H} that
	\begin{align*}
		Y_1
		& \leq e ^{\delta_2 (T_{5} - t)} \Vert \psi ( T_{5} , \theta_{ - t } \omega , \psi_0 ( \theta_{ - t } \omega ) ) \Vert _{\cH}^{2} \\
		& \leq e ^{\delta_2 (T_{5} - t)} c_1 ( r _{1} ^{ \varepsilon } ( \omega ) + 1 ).
	\end{align*}
	Thus for every given $ \zeta > 0 $, there is a $ T_{6} = T_{6} ( B , \omega , \zeta ) > T_{5} $ such that for all $ t \geq T_{6} $,
	\begin{equation}\label{n1}
		Y_1 \leq \frac{\zeta}{4}.
	\end{equation}
	Since \eqref{conclusion1} and \eqref{conclusion3} holds for all $ t \geq T_{4} $, we get that, for all $ t \geq T_{5} $,
	\begin{align*}
		Y_2
		& \leq \frac{c}{k} \int_{ T_{5} - t } ^{0} e ^{ \delta_2 s } \left[ c_1 ( r _{1} ^{ \varepsilon } ( \omega ) + 1 ) \right] \d s \\
		& \leq \frac{c}{k \delta_2} \left[ c_1 ( r _{1} ^{ \varepsilon } ( \omega ) + 1 ) \right] \left( 1 - e ^{ \delta_2 \left( T_{5} - t \right) } \right) \\
		& \leq \frac{c}{k} ( r _{1} ^{ \varepsilon } ( \omega ) + 1 ).
	\end{align*}
	Hence, there is a $ R_1 = R_1 ( \omega , \zeta ) > 0 $ such that for all $ t \geq T_{5} $ and $ k \geq R_1 \geq R_{0} $,
	\begin{equation}\label{n2}
		Y_2 \leq \frac{\zeta}{4}.
	\end{equation}
	
	Note that $ f \in \bm{H} $ and $ h(x) \in \bm{W} $, there is a $ R_{2} = R_{2} (\omega , \zeta) $ such that for all $ k \geq R_{2} $,
	\begin{equation}\label{ff}
		\Vert f(x) \Vert _{L^{2}_{\rho}}^{2} = \int_{\cO} \rho \left( \frac{ \vert x \vert ^2 }{ k^2 } \right) \vert f(x) \vert ^{2} \d x
		\leq \int_{ \vert x \vert ^2 \geq k ^2} \vert f(x) \vert ^{2} \d x
		\leq \frac{\delta_2 \zeta}{8c}
	\end{equation}
	and
	\begin{equation}\label{hh}
		\begin{aligned}
			\tilde{h} : & = \Vert h(x) \Vert _{L^{2}_{\rho}}^{2} + \Vert A^{\frac{1}{2}} h(x) \Vert _{L^{2}_{\rho}}^{2} + \Vert A h(x) \Vert _{L^{2}_{\rho}}^{2} + \Vert A^{\frac{1}{2}} h(x) \Vert _{L^{2}_{\rho}}^{4}\\
			& = \int_{\cO} \rho \left( \frac{ \vert x \vert ^2 }{ k^2 } \right) \Big( \vert h(x) \vert ^2 + \vert A^{\frac{1}{2}} h(x) \vert ^2 + \vert A h(x) \vert ^2 + \vert A^{\frac{1}{2}} h(x) \vert ^4 \Big) \d x \\
			& \leq \int_{ \vert x \vert ^2 \geq k ^2} \Big( \vert h(x) \vert ^2 + \vert A^{\frac{1}{2}} h(x) \vert ^2 + \vert A h(x) \vert ^2 + \vert A^{\frac{1}{2}} h(x) \vert ^4 \Big) \d x \\
			& \leq \min \left\{ \frac{\delta_2 \zeta}{8 \varepsilon^{2} c_{1} r(\omega) (r_{1}^{\varepsilon}(\omega) + 1)} , \frac{\delta_2 \zeta}{16 c r (\omega)} \right\},
		\end{aligned}
	\end{equation}
	where $ c $ is the constant in $ Y_{4} $.
	Therefore, by \eqref{s3}, \eqref{conclusion1} and \eqref{hh}, $ Y_3 $ is bounded by
	\begin{align*}
		Y_3
		& \leq \varepsilon ^{2} \int_{ T_{5} } ^{t} e ^{ \delta_2 ( s - t ) } \Vert A^{\frac{1}{2}} h(x) \Vert _{L^{2}_{\rho}}^{2} \abs{z ( \theta_{ s - t } \omega )} ^{4} c_{1} ( r_{1}^{\varepsilon} ( \omega ) + 1 ) \d s \\
		& \leq \frac{\delta_2 \zeta}{8 r(\omega)} \int_{ T_{5} } ^{t} e ^{ \delta_2 ( s - t ) } \vert z ( \theta_{ s - t } \omega ) \vert ^{4} \d s \\
		& = \frac{\delta_2 \zeta}{8 r(\omega)} \int_{ T_{5} - t } ^{0} e ^{ \delta_2 s } \vert z ( \theta_{ s } \omega ) \vert ^{4}\d s \\
		& \leq \frac{\delta_2 \zeta}{8 r(\omega)} \int_{ T_{5} - t } ^{0} e ^{ \frac{\delta_2}{2} s } r ( \omega ) \d s
		= \frac{\zeta}{4} \left( 1 - e ^{\frac{\delta_2}{2} ( T_{5} - t )} \right).
	\end{align*}
	Then for all $ t \geq T_{5} $,
	\begin{equation}\label{n3}
		Y_3 \leq \frac{\zeta}{4}.
	\end{equation}
	Similarly, by \eqref{s3}, \eqref{conclusion1}, \eqref{ff} and \eqref{hh}, $ Y_4 $ is bounded by
	\begin{align*}
		Y_4
		& \leq c \int_{ T_{5} } ^{t} e ^{ \delta_2 ( s - t ) } \left[ \Vert f(x) \Vert _{L^{2}_{\rho}}^{2} + \tilde{h} \beta_1 ( \theta_{ s - t } \omega ) \right] \d s \\
		& \leq \frac{\delta_2 \zeta}{8} \int_{ T_{5} } ^{t} e ^{ \delta_2 ( s - t ) } \d s
		+ \frac{\delta_2 \zeta}{16 r(\omega)} \int_{ T_{5} } ^{t} e ^{ \delta_2 ( s - t ) } \beta_1 ( \theta_{ s - t } \omega ) \d s \\
		& = \frac{\delta_2 \zeta}{8} \int_{ T_{5} - t } ^{0} e ^{ \delta_2 s } \d s
		+ \frac{\delta_2 \zeta}{16 r(\omega)} \int_{ T_{5} - t } ^{0} e ^{ \delta_2 s } \beta_1 ( \theta_{ s } \omega ) \d s \\
		& \leq \frac{\delta_2 \zeta}{8} \int_{ T_{5} - t } ^{0} e ^{ \delta_2 s } \d s
		+ \frac{\delta_2 \zeta}{16 r(\omega)} \int_{ T_{5} - t } ^{0} e ^{ \frac{\delta_2}{2} s } r ( \omega ) \d s \\
		& = \frac{\zeta}{8} \left( 1 - e ^{\delta_2 ( T_{5} - t )} \right) + \frac{\zeta}{8} \left( 1 - e ^{\frac{\delta_2}{2} ( T_{5} - t )} \right) .
	\end{align*}
	Therefore, for all $ t \geq T_{5} $,
	\begin{equation}\label{n4}
		Y_4 \leq \frac{\zeta}{4}.
	\end{equation}
	Let $ T_{7} = \max \{ T_{5} , T_{6} \} $. It follows from \eqref{H}--\eqref{n4} that, for all $ t \geq T_{7} $ and $ k \geq R_2 $,
	\begin{equation*}
		H( t , \theta_{ - t } \omega ) \leq \zeta.
	\end{equation*}
	Consequently, for all $ t \geq T_{7} $ and $ k \geq R_2 $,
	\begin{equation*}
		\Vert \psi ( t, \theta_{ - t} \omega, \psi_0( \theta_{ - t} \omega ) )(x) \Vert^2 _{ \cH ( Q_{\sqrt{2}k}^{c} ) } \leq \zeta.
	\end{equation*}
	Taking $ R_{3} = \sqrt{2} R_{2} $, we obtain \eqref{psi_zeta} which completes the proof.
\end{proof}
By using \eqref{phi}, we have next lemma.
\begin{lemma}\label{lem4}
	Suppose that $ B = \left\{ B( \omega ) \right\}_{\omega \in \Omega} \in \cD $ and $ \phi_0( \omega ) \in B( \omega ) $. Then for every $ \zeta > 0 $ and $ \bP $-$ a.e. $ $ \omega \in \Omega $, there exist a $ T_{7} = T_{7} ( B , \omega , \zeta ) > 0 $ assigned in Lemma \ref{lem3} and a $ R_{5} = R_{5} ( \omega , \zeta ) $ such that for all $ t \geq T_{7} $,
	\begin{equation}\label{phi_zeta}
		\Vert \phi ( t, \theta_{ - t} \omega, \phi_0( \theta_{ - t} \omega ) )(x) \Vert^2 _{ \cH ( Q_{R_{5}}^{c} ) } \leq \zeta.
	\end{equation}
\end{lemma}
\begin{proof}
	It can be deduced from Lemma \ref{lem3} and \eqref{phi} that for all $ t \geq T_{7} $,
	\begin{align*}
		& \quad \Vert \phi ( t, \theta_{ - t} \omega, \phi_0( \theta_{ - t} \omega ) ) \Vert^2 _{ \cH ( Q_{R_{3}}^{c} ) } \\
		& \leq 2 \Vert \psi ( t, \theta_{ - t} \omega, \psi_0( \theta_{ - t} \omega ) ) \Vert^2 _{ \cH ( Q_{R_{3}}^{c} ) }
		+ \int_{\left\vert x \right\vert \geq R_{3}} 2 \varepsilon^2 \vert y ( \omega ) \vert^2 \d x .
	\end{align*}
	Indeed, since $ y( \omega ) = h(x) z( \omega ) $ and $ h(x) \in \bm{W} $, there is a $ R_4 = R_4 ( \omega , \zeta ) > 0 $ such that
	\begin{equation*}
		\int_{\left\vert x \right\vert \geq R_4} 2 \varepsilon^2 \vert y ( \omega ) \vert^2 \d x
		\leq \zeta.
	\end{equation*}
	Hence, there exists a $ R_5 = \max \left\{ R_3 , R_4 \right\} $ such that for all $ t \geq T_{7} $,
	\begin{equation*}
		\Vert \phi ( t, \theta_{ - t} \omega, \phi_0( \theta_{ - t} \omega ) ) \Vert^2 _{ \cH ( Q_{R_{5}}^{c} ) }
		\leq 3\zeta.
	\end{equation*}
	The proof is complete.
\end{proof}
\subsection{\texorpdfstring{$ \cD $}{\mathcal{D}}--pullback asymptotic compactness}
In this section, we prove the existence and uniqueness of $ \cD $-random attractor for the random dynamical system $ \Phi $ corresponding to the stochastic Navier-Stokes equation with memory in unbounded domains. To overcome the difficulty caused by the lack of compactness of $ \mathcal{M}^{1} \hookrightarrow \mathcal{M} $ (see e.g., \cite{PataZ2001}), we construct a new compact subspace as in \cite{Liu2017a}.


First, we give a lemma on producing a compact subspace $ \mathcal{N} \subset \mathcal{M} $.
\begin{lemma}\label{N}
	Denote by
	$$
		\mathcal{N}
		=
		\bigcup_{ \eta_0 \in K( \theta_{ - t } \omega ) }
		\bigcup_{ t \geq T_{7} }
		\bigcup_{ \omega \in \Omega }
		\eta^{t} ( \theta_{ - t } \omega , \eta_0 ),
	$$
	where $ \left\{ K( \omega ) \right\}_{ \omega \in \Omega } $ is defined by \eqref{K_omega} and $ T_{7} $ is defined in Lemma \ref{lem3}.
	Then $ \mathcal{N} $ is relatively compact in $ \mathcal{M} $.
\end{lemma}
\begin{proof}
	See \cite{GP2000,Liu2017a}.
\end{proof}

Next, we follow the procedure in \cite{BLW2009} with Lemma \ref{lem2}, Lemma \ref{lem4} and Lemma \ref{N} to get the $ \cD $-pullback asymptotic compactness of $ \Phi $.
\begin{lemma}\label{compactness}
	The random dynamical system $ \Phi $ is $ \cD $-pullback asymptotically compact in $ \cH (\cO) $; that is, for $ \bP $-a.e. $ \omega \in \Omega $, the sequence $ \left\{ \Phi ( t_{n} , \theta_{ - t_{n} } \omega ) \phi_{0,n} ( \theta_{ - t_{n} } \omega ) \right\} $ has a convergent subsequence in $ \cH ( \cO ) $ provided $ t_{n} \rightarrow \infty $, $ B = \left\{ B ( \omega ) \right\}_{ \omega \in \Omega } \subset \cD $ and $ \phi_{0,n} ( \theta_{ - t_{n} } \omega ) \in B ( \theta_{ - t_{n} } \omega ) $.
\end{lemma}
\begin{proof}
	 For $ t_{n} \rightarrow \infty $, $ B = \left\{ B ( \omega ) \right\}_{ \omega \in \Omega } \in \cD $ and $ \phi_{0,n} ( \theta_{ - t_{n} } \omega ) \in B ( \theta_{ - t_{n} } \omega ) $,
	 from Lemma \ref{lem1}, we have that
	\begin{equation*}
		\left\{ \Phi ( t_{n} , \theta_{ - t_{n} } \omega ) \phi_{0,n} ( \theta_{ - t_{n} } \omega ) \right\}_{ n = 1 }^{ \infty } \mbox{ is bounded in } \cH ( \cO ).
	\end{equation*}
	Then there exists $ \hat{\phi} \in \cH ( \cO ) $ such that, up to a subsequence,
	\begin{equation}\label{weakcon}
		\Phi ( t_{n} , \theta_{ - t_{n} } \omega ) \phi_{0,n} ( \theta_{ - t_{n} } \omega ) \rightarrow \hat{\phi} \quad \mbox{ weakly in } \cH ( \cO ).
	\end{equation}
	
	In what follows, we prove that the weak convergence in \eqref{weakcon} is actually strong convergence, for which we need to prove
	\begin{equation*}
		\Vert \phi ( t_{n} , \theta_{ - t_{n} } \omega, \phi_{ 0, n }( \theta_{ - t_{n} } \omega ) ) - \hat{\phi} \Vert _{\cH ( \cO )}^2
		\leq \zeta,
	\end{equation*}
	for a given $ \zeta > 0 $. Lemma \ref{lem4} implies that there exist a $ T_{7} = T_{7} ( B , \omega , \zeta ) $ and a $ R_{5} = R_{5} ( \omega , \zeta ) $, such that for all $ t \geq T_{7} $,
	\begin{equation}\label{phi_R3}
		\Vert \phi ( t, \theta_{ - t} \omega, \phi_0( \theta_{ - t} \omega ) ) \Vert^2 _{ \cH ( Q_{R_{5}}^{c} ) } \leq \frac{\zeta}{8}.
	\end{equation}
	Since $ t_{n} \rightarrow \infty $, there is a $ N_{1} = N_{1} ( B , \omega , \zeta ) \in \mathbb{N} $ such that for all $ n \geq N_{1} $, $ t_{n} \geq T_{7} $, it can be deduced from \eqref{phi_R3} that
	\begin{equation}\label{phi_R3n}
		\Vert \phi ( t_{n} , \theta_{ - t_{n} } \omega, \phi_{ 0, n }( \theta_{ - t_{n} } \omega ) ) \Vert^2 _{ \cH ( Q_{R_{5}}^{c} ) } \leq \frac{\zeta}{8}.
	\end{equation}
	Notice that $ \hat{\phi} \in \cH ( \cO ) $ and indeed there exists a $ R_{6} = R_{6} ( \zeta ) > 0$ such that
	\begin{equation}\label{varphi}
	\Vert \hat{\phi} (x) \Vert^2 _{ \cH ( Q_{R_{6}}^{c} ) } \leq \frac{\zeta}{8}.
	\end{equation}
	
	It remains to be done with the problem that $ \mathcal{M}^1 ( Q_{R} ) \hookrightarrow \mathcal{M} ( Q_{R} ) $ is not compact. Let $ R_{7} = \max \left\{ R_{5} , R_{6} \right\} $, by Lemma \ref{N}, we know that $ \overline{\mathcal{N}} $ is compact in $ \mathcal{M} $ where $ \overline{\mathcal{N}} $ is the closure of $ \mathcal{N} $.
	Define $ \widetilde{\cH} : = \bm{V} \times ( \mathcal{M}^{1} \cap \overline{\mathcal{N}} ) \subset \cH^{1} \subset \cH $, then it follows from \eqref{conclusion4} that for all $ t \geq T_{7} $,
	\begin{equation*}
		\Vert \phiN ( t , \theta_{ - t } \omega , \phi_{0N} ( \theta_{ - t } \omega ) ) \Vert _{ \widetilde{\cH} ( \cO ) } ^2
		\leq c_4 (r_{4}^{\varepsilon}(\omega) + 1).
	\end{equation*}
	Hence, there is a $ N_{2} = N_{2} ( B , \omega) \in \mathbb{N} $ large enough such that for all $ n \geq N_{2} $, $ t_{n} \geq T_{7} $,
	\begin{equation*}
		\Vert \phiN ( t_{n} , \theta_{ - t_{n} } \omega , \phi_{0N,n} ( \theta_{ - t_{n} } \omega ) ) \Vert _{ \widetilde{\cH} ( \cO ) } ^2
		\leq c_4 (r_{4}^{\varepsilon}(\omega) + 1).
	\end{equation*}
	With compact embedding $ \bm{V} ( Q_{R_{7}} ) \hookrightarrow \bm{H} ( Q_{R_{7}} ) $ and $ ( \mathcal{M}^{1} ( Q_{R_{7}} ) \cap \overline{\mathcal{N}} ( Q_{R_{7}} ) ) \subset \overline{\mathcal{N}} ( Q_{R_{7}} ) \hookrightarrow \mathcal{M} ( Q_{R_{7}} ) $, we know that
	$ \widetilde{\cH} ( Q_{R_{7}} ) $ is compact in $ \cH ( Q_{R_{7}} ) $.
	Consequently, up to a subsequence, we get that
	\begin{equation*}
		\phiN ( t_{n} , \theta_{ - t_{n} } \omega ) \phi_{0N,n} ( \theta_{ - t_{n} } \omega ) \rightarrow \hat{\phi} \quad \mbox{ strongly in } \cH ( Q_{R_{7}} ),
	\end{equation*}
	which indicates that for a given $ \zeta > 0 $, there exists a $ N_{3} = N_{3} ( B , \omega , \zeta ) \in \mathbb{N} $ such that for all $ n \geq N_{3} $,
	\begin{equation}\label{phi_varphi}
		\Vert \phiN ( t_{n} , \theta_{ - t_{n} } \omega, \phi_{ 0N, n }( \theta_{ - t_{n} } \omega ) ) - \hat{\phi} \Vert _{\cH ( Q_{R_{7}} )}^2 \leq \frac{\zeta}{8}.
	\end{equation}
	It can be deduced from the Lemma \ref{phiL} that for a given $ \zeta > 0 $, there are a $ N_{4} ( B , \omega , \zeta ) \in \mathbb{N} $ sufficiently large and a $ T_{8} > T_{7} $ such that for $ n \geq N_{4} $, we have $ t_{n} > T_{8} $ and
	\begin{equation}\label{phiL2}
		\norm{\phiL ( t_{n} , \theta_{ - t_{n} } \omega, \phi_{ 0L, n }( \theta_{ - t_{n} } \omega ) )}_{ \cH ( Q_{R_{7}} ) }^{2} \leq \frac{\zeta}{8}.
	\end{equation}
	
	Set $ N_{5} = \max \left\{ N_{1} , N_{3} , N_{4} \right\} $, it follows from \eqref{phi_R3n} -- \eqref{phiL2} that for all $ n \geq N_{5} $, $ t_{n} \geq T_{8} $,
	\begin{align*}
		& \Vert \phi ( t_{n} , \theta_{ - t_{n} } \omega, \phi_{ 0, n }( \theta_{ - t_{n} } \omega ) ) - \hat{\phi} \Vert _{\cH ( \cO )}^2\\
		& \quad \leq \Vert \phi ( t_{n} , \theta_{ - t_{n} } \omega, \phi_{ 0, n }( \theta_{ - t_{n} } \omega ) ) - \hat{\phi} \Vert^2 _{ \cH ( Q_{R_{7}} ) } \\
		& \quad \quad + \Vert \phi ( t_{n} , \theta_{ - t_{n} } \omega, \phi_{ 0, n }( \theta_{ - t_{n} } \omega ) ) - \hat{\phi} \Vert^2 _{ \cH ( Q_{R_{7}}^{c} ) } \\
		& \quad \leq 2 \Vert \phiN ( t_{n} , \theta_{ - t_{n} } \omega, \phi_{ 0N, n }( \theta_{ - t_{n} } \omega ) ) - \hat{\phi} \Vert^2 _{ \cH ( Q_{R_{7}} ) } \\
		& \quad \quad + 2 \norm{\phiL ( t_{n} , \theta_{ - t_{n} } \omega, \phi_{ 0L, n }( \theta_{ - t_{n} } \omega ) )}_{ \cH ( Q_{R_{7}} ) }^{2} \\
		& \quad \quad + 2 \Vert \phi ( t_{n} , \theta_{ - t_{n} } \omega, \phi_{ 0, n }( \theta_{ - t_{n} } \omega ) ) \Vert^2 _{\cH ( Q_{R_{7}}^{c} )} \\
		& \quad \quad + 2 \Vert \hat{\phi} \Vert^2 _{ \cH ( Q_{R_{7}}^{c} ) } \\
		& \quad \leq \zeta.
	\end{align*}
	This completes the proof.
\end{proof}
Since \eqref{K_omega} implies a closed random absorbing set $ \left\{ K ( \omega ) \right\} _{ \omega \in \Omega } $ for $ \Phi $, and $ \Phi $ is $ \cD $-pullback asymptotically compact in $ \cH (\cO) $ from Lemma \ref{compactness}, we immediately get the following theorem by Proposition \ref{attractors}.
\begin{theorem}
	The random dynamical system $ \Phi $ has a unique $ \cD $-random attractor in $ \cH ( \cO ) $.
\end{theorem}

\section{Upper semicontinuity of random attractors}\label{upper}
In this section, we prove the upper semicontinuity of random attractors for Navier-Stokes equations with memory in unbounded domains by Proposition \ref{upperc} with the constructed compact embedding and the unifom estimates above. When $ \varepsilon \rightarrow 0 $, the limiting deterministic system of \eqref{SNSM4} is
\begin{equation}\label{DNSM}
	\left\{
	\begin{aligned}
		& \pt_t u + \nu Au + \int_{0}^{\infty} \mu (s) A \eta (s)\d s + B( u , u ) = f, \\
		& \pt_t \eta = T \eta + u,\\
		& u(0) = u_0,\ \eta^0 = \eta_0.
	\end{aligned}
	\right.
\end{equation}

\begin{lemma}\label{lem5}
	For a given $ 0 < \varepsilon \leq 1 $, let $ \psi^{\varepsilon} = ( v^{\varepsilon} , \eta^{\varepsilon} )^{\tau} $ and $ \phi = ( u , \eta )^{\tau} $ be the solution of \eqref{SNSM4} and \eqref{DNSM} with initial conditions $ \psi_{0}^{\varepsilon} = ( v^{\varepsilon}_{0} , \eta^{\varepsilon}_{0} )^{\tau} $ and $ \phi_{0} = ( u_{0} , \eta_{0} )^{\tau} $, respectively. Then for $ \bP $-a.e. $ \omega \in \Omega $ and $ t \geq T_{4} $, we have
	$$
		\begin{aligned}
		\Vert \psi^{\varepsilon} ( t , \omega , \psi_{0}^{ \varepsilon } ) -  \phi ( t , \omega , \phi_{0} ) \Vert _{\cH} ^{2}
		& \leq c e ^{ c t } \Vert \psi_{0}^{ \varepsilon } - \phi_{0} \Vert _{\cH} ^{2} + c \varepsilon^{2} e^{ c t } r(\omega) ( 1 + r_{1}^{\varepsilon} (\omega) + r_{3}^{\varepsilon} (\omega) ),
		\end{aligned}
	$$
	where $ c > 0 $ is independent of $ \varepsilon $.
\end{lemma}
\begin{proof}
	Let $ \varphi = ( w , \xi )^{\tau} = \psi^{\varepsilon} - \phi $. Since $ \psi^{\varepsilon} $, $ \phi $ satisfy \eqref{SNSM4} and \eqref{DNSM}, respectively, we deduce that
	\begin{equation}\label{W}
		\left\{
			\begin{aligned}
			&\pt_t w + \nu Aw + \int_{0}^{\infty} \mu (s) A \xi (s)\d s + B( v^{\varepsilon} + \varepsilon y , v^{\varepsilon} + \varepsilon y ) - B( u , u ) = \varepsilon ( \sigma y - \nu A y ), \\
			&\pt_t \xi = T \xi + w + \varepsilon y,\\
			&w(0) = w_{0} = v_{0}^{\varepsilon} - u_0,\ \xi_{0} = \eta_0^{\varepsilon} - \eta_0.
		\end{aligned}
		\right.
	\end{equation}
	Taking the inner product with \eqref{W} by $ \varphi $ in $ \cH $, we have
	\begin{equation}\label{WE}
		\begin{aligned}
			\frac{1}{2} \frac{\d}{\d t} \left( \Vert w \Vert^2 + \Vert \xi \Vert ^2 _{\mathcal{M}} \right) + \nu \Vert A^{\frac{1}{2}} w \Vert^2
			& = \frac{1}{2} \int_{0}^{\infty} \mu'(s) \Vert A^{\frac{1}{2}} \xi \Vert ^2 \d s + \varepsilon (\xi, y(\theta_t \omega))_{\mathcal{M}} \\
			& \quad + b(u , u , w) - b(v ^{\varepsilon} + \varepsilon y(\theta_t \omega), v ^{\varepsilon} + \varepsilon y(\theta_t \omega) , w)\\
			& \quad + ( \varepsilon \left( \sigma y(\theta_t \omega) - \nu A y(\theta_t \omega)\right) , w ).\\
		\end{aligned}
	\end{equation}
	The same procedure as in Lemma \ref{lem1} yields
	\begin{equation}\label{E1w}
		\begin{aligned}
			\frac{\d}{\d t} \Vert \varphi (t,\omega,\varphi_0(\omega))\Vert^2_{\cH}
			& \leq \mathcal{Q}(t,\omega) \Vert \varphi (t,\omega,\varphi_0(\omega)) \Vert^2 _{\cH} + \mathcal{R}(t , \omega),
		\end{aligned}
	 \end{equation}
	 where
	 \begin{align*}
	 	& \mathcal{Q}(t,\omega) := \left( - \delta_0 + c\varepsilon^{2} \Vert \nabla v^{\varepsilon} \Vert^{2} \right), \\
		& \mathcal{R}(t , \omega) := c \varepsilon^{2} \beta_{1}(\theta_{ t } \omega) \left( 1 + \Vert v^{\varepsilon} \Vert^{2} + \Vert \nabla v^{\varepsilon} \Vert^{2} \right).
	 \end{align*}
	 Applying Gronwall's inequality, we obtain
	 \begin{equation}\label{phiw}
		\Vert \varphi (t,\omega,\varphi_0(\omega)) \Vert^2 _{\cH}
		\leq e^{ \int_{0}^{t} \mathcal{Q}(\tau, \omega) \d \tau } \Vert \varphi_0 (\omega) \Vert^2 _{\cH}
		 + c \int_{0}^{t} e^{ \int_{s}^{t} \mathcal{Q}(\tau, \omega) \d \tau } \mathcal{R} ( s , \omega) \d s .
	 \end{equation}
	 By means of \eqref{conclusion1}, we have
	 $$
	 	\int_{0}^{t} \mathcal{Q}(\tau,\omega) \d \tau \leq c t.
	 $$
	 It follows from \eqref{s3} and Lemma \ref{lem2} that for all $ t \geq T_{4} $,
	 $$
	 	\int_{0}^{t} e^{ \int_{s}^{t} \mathcal{Q}(\tau, \omega) \d \tau } \mathcal{R} ( s , \omega) \d s
	 	\leq c \varepsilon^{2} e^{ c t } r(\omega) ( 1 + r_{1}^{\varepsilon} (\omega) + r_{3}^{\varepsilon} (\omega) ).
	 $$
	 This completes the proof.
\end{proof}
Let $ \phi ^{\varepsilon} $ be the solution of \eqref{SNSM3} with initial conditions $ \phi_0^{\varepsilon} = ( u_{0}^{\varepsilon} , \eta_{0}^{\varepsilon} ) $ and $ \Phi ^{\varepsilon} $ be the corresponding cocycle. By \eqref{phi}, one obtains that
\begin{align*}
	\Vert \Phie ( t , \omega ) \phi_0^{\varepsilon} - \phi ( t , \omega , \phi_{0} ) \Vert _{\cH}^{2}
	& = \Vert \phi ^{\varepsilon}( t , \omega , \phi_0^{\varepsilon} ) - \phi ( t , \omega , \phi_{0} ) \Vert _{\cH}^{2} \\
	& = \Vert \psi ^{\varepsilon} ( t , \omega , \psi_{0}^{\varepsilon} ) + ( \varepsilon y (\theta_{ t } \omega) , 0 ) - \phi ( t , \omega , \phi_{0} ) \Vert _{\cH}^{2} \\
	& \leq 2 \Vert \psi^{\varepsilon} ( t , \omega , \psi_{0}^{ \varepsilon } ) -  \phi ( t , \omega , \phi_{0} ) \Vert _{\cH} ^{2} + 2 \varepsilon^{2} e^{\frac{\delta_0}{2} t} r( \omega ).
\end{align*}
Then, we immediately derive the following lemma which satisfies the first condition of Proposition \ref{upperc}.
\begin{lemma}\label{lem6}
	For a given $ 0 < \varepsilon \leq 1 $, assume that conditions on Lemma \ref{lem5} hold, then for $ \bP $-a.e. $ \omega \in \Omega $ and $ t \geq T_{4} $, we have
	$$
		\begin{aligned}
			\Vert \phi ^{\varepsilon} ( t , \omega , \phi_0^{\varepsilon} ) - \phi ( t , \omega , \phi_{0} ) \Vert _{\cH}^{2}
			& \leq c e ^{ c t } \Vert \phi_{0}^{ \varepsilon } - \phi_{0} \Vert _{\cH} ^{2} + c \varepsilon^{2} e^{ c t } r(\omega) ( 1 + r_{1}^{\varepsilon} (\omega) + r_{3}^{\varepsilon} (\omega) ),
		\end{aligned}
	$$
	where $ c > 0 $ is independent of $ \varepsilon $.
\end{lemma}

Next, we prove the third condition in Proposition \ref{upperc}.
Given $ 0 < \varepsilon \leq 1 $, from the proof of Lemma \ref{lem2}, we rewrite \eqref{K_omega} as
\begin{equation}
	K_{\varepsilon}(\omega) = \left\{ \phi \in \cH(\cO) : \Vert \phi \Vert ^2 _{\cH} \leq c_2 ( r _{2} ^{ \varepsilon }(\omega) + 1 ) \right\},
\end{equation}
that is, for every $ B = \{ B(\omega) \}_{\omega \in \Omega} \in \cD $ and $ \bP $-a.e. $ \omega \in \Omega $, there is a $ T_{3} > 0 $, independent of $ \varepsilon $, such that for all $ t \geq T_{3} $,
$$
	\Vert \Phie ( t , \theta_{ - t } \omega ) B( \theta_{ - t } \omega ) \Vert _{\cH}^{2} \leq c_2 ( r _{2} ^{ \varepsilon }(\omega) + 1 ).
$$
We also denote
\begin{equation}
	K(\omega) = \left\{ \phi \in \cH(\cO) : \Vert \phi \Vert ^2 _{\cH} \leq c_2 ( r _{2} ^{ 1 }(\omega) + 1 ) \right\}.
\end{equation}
To obtain the compactness in the next lemma, we define two sets $ \wtK_{\varepsilon}(\omega) = K_{\varepsilon}(\omega) \cap \widetilde{\cH} $ and $ \wtK(\omega) = K(\omega) \cap \widetilde{\cH} $ where $ \widetilde{\cH} $ is assigned in Lemma \ref{compactness}. Then $ \{ \wtK_{\varepsilon}(\omega) \}_{\omega \in \Omega} $ is a closed absorbing set for $ \Phie $ in $ \cD $ and
$$
	\bigcup_{0 < \varepsilon \leq 1} \wtK_{\varepsilon}(\omega) \subseteq \wtK(\omega).
$$
It follows from the invariance of the random attractor $ \left\{ \cA_{\varepsilon}(\omega) \right\}_{\omega \in \Omega} $ that
\begin{equation}\label{A_e}
	\bigcup_{0 < \varepsilon \leq 1} \cA_{\varepsilon}(\omega) \subseteq \bigcup_{0 < \varepsilon \leq 1} \wtK_{\varepsilon} (\omega)
	\subseteq \wtK(\omega).
\end{equation}
From \eqref{linear} and \eqref{nonlinear} we know that $ \Phie $ can be written as $ \Phie = \Phie_{N} + \Phie_{L} $. So we denote $ \cA_{\varepsilon}(\omega) $ by $ \cA_{\varepsilon}(\omega) = \cA_{N\varepsilon}(\omega) \cup \cA_{L\varepsilon}(\omega) $ where $ \cA_{N\varepsilon}(\omega) $ and $ \cA_{L\varepsilon}(\omega) $ are generated by $ \Phie_{N} $ and $ \Phie_{L} $ respectively.
Consequently, by Lemma \ref{lem2}, there exists $ T_{4} > 0 $, independent of $ \varepsilon $, such that for all $ t \geq T_{4} $,
$$
	\Vert \Phie_{N}(t , \theta_{ - t } \omega) \cA_{N\varepsilon} (\theta_{ - t } \omega) \Vert_{ \widetilde{\cH} ( \cO ) } ^2
	\leq c_{4}(r _{4} ^{ \varepsilon }(\omega) + 1)
	\leq c_{4}(r _{4} ^{ 1 }(\omega) + 1).
$$
According to the invariance of attractor, i.e. $ \cA _{\varepsilon} (\omega) = \Phie(t , \theta_{ - t } \omega) \cA_{\varepsilon} (\theta_{ - t } \omega) $ for all $ t \geq 0 $, $ \bP $-a.e. $ \omega \in \Omega $, we have that for every $ \phiN \in \bigcup_{0 < \varepsilon \leq 1} \cA _{N\varepsilon} (\omega) $, $ \bP $-a.e. $ \omega \in \Omega $,
\begin{equation}\label{phi_H}
	\Vert \phiN \Vert _{\widetilde{\cH} (\cO) } ^2
	\leq c_{4}(r _{4} ^{ 1 }(\omega) + 1).
\end{equation}

\begin{lemma}\label{lem7}
	The union $ \bigcup_{0 < \varepsilon \leq 1} \cA_{\varepsilon} ( \omega ) $ is precompact in $ \cH ( \cO ) $ for every $ \omega \in \Omega $.
\end{lemma}

We will prove Lemma \ref{lem7} by the method in \cite{Wang2009} with some modifications in which we use our constructed compact embedding $ \widetilde{\cH} \hookrightarrow \cH $ in a bounded ball from the proof of Lemma \ref{compactness}.
\begin{proof}[Proof of Lemma \ref{lem7}]
	To show that $ \bigcup_{0 < \varepsilon \leq 1} \cA_{\varepsilon} ( \omega ) $ is precompact in $ \cH ( \cO ) $, that is, given $ \zeta > 0 $, $ \bigcup_{0 < \varepsilon \leq 1} \cA_{\varepsilon} ( \omega ) $ has a finite covering of balls of radii less than $ \zeta $, we establish the far-field estimate and the bounded ball estimate by the compact embedding we construct.
	
	Let $ \{ \wtK_{\varepsilon}(\omega) \}_{\omega \in \Omega} $ be the random set given before. It follows from Lemma \ref{lem4} that given $ \zeta > 0 $, $ \bP $-a.e. $ \omega \in \Omega $, there exist $ T_{7} > 0 $ and $ R_{5} > 0 $, such that for all $ t \geq T_{7} $ and $ \phi_{0N}^{\varepsilon} ( \theta_{ - t } \omega ) \in \wtK ( \theta_{ - t } \omega ) $,
	\begin{equation*}
		\Vert \Phie_{N} ( t , \theta_{ - t } \omega ) \phi_{0N}^{\varepsilon} ( \theta_{ - t } \omega ) \Vert^2 _{ \cH ( Q_{R_{5}}^{c} ) }
		\leq \frac{\zeta ^2}{64}.
	\end{equation*}
	Since \eqref{A_e} holds, $ \phi_{0N}^{\varepsilon} ( \theta_{ - t } \omega ) \in \cA _{N\varepsilon} ( \theta_{ - t } \omega ) $ implies $ \phi_{0N}^{\varepsilon} ( \theta_{ - t } \omega ) \in \wtK ( \theta_{ - t } \omega ) $. Hence for every $ 0 < \varepsilon \leq 1 $, $ \bP $-a.e. $ \omega \in \Omega $, $ t \geq T_{7} $ and $ \phi_{0N}^{\varepsilon} ( \theta_{ - t } \omega ) \in \cA _{N\varepsilon} ( \theta_{ - t } \omega ) $, we have
	\begin{equation*}
		\Vert \Phie_{N} ( t , \theta_{ - t } \omega ) \phi_{0N}^{\varepsilon} ( \theta_{ - t } \omega ) \Vert^2 _{ \cH ( Q_{R_{5}}^{c} ) }
		\leq \frac{\zeta ^2}{64}.
	\end{equation*}
	Due to the invariance of $ \{ \cA _{\varepsilon} (\omega) \} _{\omega \in \Omega} $, it is clear that for $ \bP $-a.e. $ \omega \in \Omega $ and for all $ \phiN \in \bigcup_{0 < \varepsilon \leq 1} \cA_{N\varepsilon} ( \omega ) $,
	\begin{equation*}
		\Vert \phiN(x) \Vert^2 _{ \cH ( Q_{R_{5}}^{c} ) }
		\leq \frac{\zeta ^2}{64},
	\end{equation*}
	which is equivalent to
	\begin{equation}\label{phi_Hc}
		\Vert \phiN(x) \Vert _{\cH ( Q_{R_{5}}^{c} )}
		\leq \frac{\zeta}{8}.
	\end{equation}
	
	As \eqref{phi_H} indicates that $ \bigcup_{0 < \varepsilon \leq 1} \cA_{N\varepsilon} ( \omega ) $ is bounded in $ \widetilde{\cH} ( Q_{R_{5}} ) $ for $ \bP $-a.e. $ \omega \in \Omega $, by the compactness of embedding $ \widetilde{\cH} ( Q_{R_{5}} ) \hookrightarrow \cH ( Q_{R_{5}} ) $, we know that, for a given $ \zeta $, $ \bigcup_{0 < \varepsilon \leq 1} \cA_{N\varepsilon} ( \omega ) $ has finite covering of balls of radii less than $ \frac{\zeta}{8} $ in $ \cH ( Q_{R_{5}} ) $.
	
	On the other hand, Remark \ref{phiL}, \eqref{A_e} and the invariance of $ \cA_{\varepsilon} $ imply that given $ \zeta > 0 $, we have
	$$
		\norm{\phiL(x)}_{\cH (Q^c_{R_{5}})} \leq \frac{\zeta}{8},
	$$
	which means that $ \bigcup_{0 < \varepsilon \leq 1} \cA_{L\varepsilon} ( \omega ) $ has finite covering of balls of radii less than $ \frac{\zeta}{8} $ in $ \cH ( Q_{R_{5}} ) $. Then $ \bigcup_{0 < \varepsilon \leq 1} \cA_{\varepsilon} ( \omega ) $ has finite covering of balls of radii less than $ \frac{\zeta}{4} $ in $ \cH ( Q_{R_{5}} ) $.
	
	In the end, $ \bigcup_{0 < \varepsilon \leq 1} \cA_{\varepsilon} ( \omega ) $ has a finite covering of balls of radii less than $ \zeta $ in $ \cH (\cO) $.
\end{proof}
\begin{theorem}
	For $ 0 < \varepsilon \leq 1 $, the family of random attractors $ \left\{ \cA_{\varepsilon} (\omega) \right\}_{\omega \in \Omega} $ is upper semicontinuous at $ \varepsilon = 0 $, i.e., for $ \bP $-a.e. $ \omega \in \Omega $,
	$$
		\lim_{\varepsilon \rightarrow 0} \dist ( \cA_{\varepsilon}(\omega) , \cA_{0} ) = 0.
	$$
\end{theorem}
\begin{proof}
	We know that $ \{ \wtK_{\varepsilon}(\omega) \}_{\omega \in \Omega} $ is a closed absorbing set for $ \Phie $ in $ \cD $. From the definition of $ \wtK_{\varepsilon}(\omega) $, we find that
	\begin{equation}\label{Kbound}
		\begin{aligned}
			\limsup _{\varepsilon \rightarrow 0} \Vert \wtK_{\varepsilon}(\omega) \Vert
			& \leq c_{2} ( r _{1}^{0} (\omega) + 1) \\
			& \leq c_{2} \left( \int_{ - \infty}^{0} e ^{\delta_{0} s} \d s + 1 \right) \\
			& \leq c_{2} \left( \frac{1}{\delta_{0}} + 1 \right).
		\end{aligned}
	\end{equation}
	By Lemma \ref{lem6}, we have that for $ \bP $-a.e. $ \omega \in \Omega $, and $ t \geq T_{4} $,
	\begin{equation}\label{Phiconv}
		\Phi^{\varepsilon_{n}} ( t , \omega ) \phi_{0,n} \rightarrow \Phi ( t ) \phi_{0},
	\end{equation}
	provided $ \varepsilon_{n} \rightarrow 0 $ and $ \phi_{0,n} \rightarrow \phi_{0} $ in $ \cH(\cO) $.
	Since \eqref{Phiconv}, \eqref{Kbound} and Lemma \ref{lem7} verify three conditions in Proposition \ref{upperc}, respectively, we obtain the upper continuity of random attractors for Navier-Stokes equations with memory in unbounded domains.
\end{proof}

\section*{Acknowledgements}
The authors express their  sincere thanks to Prof. Guangying Lv for his constructive comments
and suggestions that helped to improve this paper.
This work was supported by the National Natural Science Foundation of China [grant number 12271261], the Key Research and Development Program of Jiangsu Province (Social Development) [grant number BE2019725],   and the Qing Lan Project of Jiangsu Province.

\end{document}